\documentclass[12pt]{amsart}
\usepackage{eurosym}
\usepackage{amsmath}
\usepackage{amsfonts}
\usepackage{graphicx}
\usepackage{caption}
\usepackage{subcaption}
\usepackage{verbatim}
\usepackage{fullpage}

\usepackage[dvipsnames]{xcolor}
\newcommand*\gf[1]{\textcolor{RoyalBlue}{\textbf{*Gary: #1*}}}

\newtheorem{theorem}{Theorem}

\newtheorem{corollary}[theorem]{Corollary}

\newtheorem{definition}[theorem]{Definition}

\newtheorem{lemma}[theorem]{Lemma}

\newtheorem{proposition}[theorem]{Proposition}

\begin{document}
\title{Optimal linear response for expanding circle maps}
\author{Gary Froyland$^{1}$}
\email{$^{1}$ g.froyland@unsw.edu.au}
\address{School of Mathematics and Statistics, UNSW Sydney, Sydney NSW 2031, Australia}
\author{Stefano Galatolo$^{2}$}
\email{$^{2}$ stefano.galatolo@unipi.it}
\address{Dipartimento di Matematica, Universita di Pisa, Via \ Buonarroti
1,Pisa - Italy}

\begin{abstract}
We consider the problem of optimal linear response for deterministic expanding maps of the circle. 
To each infinitesimal perturbation $\dot{T}$ of a circle map $T$ we consider (i) the response of the expectation of an observation function and (ii) the response of isolated spectral points of the transfer operator of $T$.
In each case, under mild conditions on the set of feasible perturbations $\dot{T}$ we show there is a unique optimal feasible infinitesimal perturbation $\dot{T}_{\rm optimal}$, maximising the increase of the  expectation of the given observation function or maximising the increase of the spectral gap of the transfer operator associated to the system. We derive expressions for the unique maximiser $\dot{T}_{\rm optimal}$ in terms of its Fourier coefficients. We also devise a Fourier-based computational scheme and apply it to illustrate our theory.


\end{abstract}

\maketitle

\section{Introduction}
A $C^3$ 
uniformly expanding map of the circle $T:S^1\to S^1$ is well known to display a linear response to its unique invariant density $f_0$ .
That is, differentiable changes to the map $T$ lead to differentiable changes in $f_0$ (see \cite{Ba2} for a survey on the subject).
Classically, linear response is often phrased as the differentiability of the expectation $\int_{S^1} c(x)\cdot f_0(x)\ dx$ of an observable $c:S^1\to\mathbb{R}$.
Each infinitesimal perturbation $\dot{T}\in C^3$ of $T$ leads to an infinitesimal perturbation $R(\dot{T})$ of $f_0$.
If $\dot{T}$ is constrained in some meaningful way, for a particular observation $c$, it is natural to ask whether there is a perturbation $\dot{T}$ that maximises $\int_{S^1} c(x)\cdot R(\dot{T})\ dx$, and if so, whether such a $\dot{T}$ is unique.
For $c\in L^\infty$ we show that there is in fact a unique maximiser under mild conditions on the feasible set of perturbations.
When the perturbations $\dot{T}$ are norm-constrained, e.g.\ by a Sobolev norm $\|\dot{T}\|_{H^4}$ we derive a relatively explicit formula for the unique maximiser $\dot{T}$ in terms of Fourier coefficients. 

We pose a similar question for the effect of perturbations on the isolated spectrum of the transfer operator $L$ of $T$. 
Perturbations $\dot{T}$ of $T$ lead to perturbations of $L$, which in turn lead to perturbations of the isolated spectrum and associated eigenprojections.
If $\dot{T}$ is constrained in a meaningful way, it is natural to ask whether there is a $\dot{T}$ that maximises the rate of change the magnitude of an isolated spectrum point $\lambda_0$.
If the isolated spectral point $\lambda_0$ is the largest-magnitude spectral point inside the unit circle, it controls the exponential rate of mixing of the system.
Therefore one can phrase this spectral optimisation question as ``does there exist a perturbation $\dot{T}$ that maximises the infinitesimal change $\dot{\lambda}$ in the mixing rate?''.
In order to answer such quantitative questions, we derive an expression for $\dot{\lambda}$ (the derivative of $\lambda$ with respect to the perturbation $\dot{T}$) in terms of $\dot{T}$. 
Under mild conditions on the feasible set of perturbations we show that there is a unique maximiser $\dot{T}$ and when this feasible set is norm constrained in e.g. $H^4$, we construct an explicit formula for the optimal $\dot{T}$ in terms of its Fourier coefficients.

To numerically estimate the unique perturbation $\dot{T}$ that maximises the expected response of a given observable $c$, we devise a Fourier-based numerical scheme.
This scheme estimates the transfer operator $L$, the action of the resolvent $(I-L)^{-1}$, and all other derivatives and integrals involved in computing the Fourier coefficients of the unique maximiser $\dot{T}$.
To numerically estimate the unique perturbation maximally affecting the mixing rate of the dynamics, we use a related Fourier scheme.
In addition to estimating the transfer operator $L_0$ and its outer spectrum when acting on $W^{1,1}(S^1)$, we also numerically approximate the eigenvector $v_0$ corresponding to the selected isolated spectral value $\lambda_0$, and a representative $\phi_0\in H^1$ of the  corresponding adjoint eigenfunctional $\varphi_0\in (W^{1,1}(S^1))^*$ of $L_0^*$ acting on $(W^{1,1}(S^1))^*$.
Each of the above terms are crucial pieces of the quantitative expression for the objective function we optimise.

Our theory and numerics are illustrated in two examples.
In the first example we consider a circle map $T$ with a slightly ``sticky'' (derivative near to 1) fixed point at $x=0$.
We show that if the observation $c$ takes large values at $x=0$, the perturbation $\dot{T}$ retains the fixed point at $x=0$ and further reduces the derivative, making it more sticky.
This increases the proportion of time that orbits spend near $x=0$ and increases the expectation of the observable.
We then show that if the observation function $c$ takes on its maximal value away from $x=0$, the optimal perturbation $\dot{T}$ sharply moves the fixed point from $x=0$ in an attempt to weight the invariant density toward larger values of $c$.

In the second example we construct a circle map with a positive isolated spectral value larger than $1/\inf_x|T'(x)|$.
The presence of the relatively large isolated spectral value is due to almost-invariant intervals to the right of $x=0$ and to the left of $x=1$.
We show that the perturbation $\dot{T}$ that maximally slows the mixing rate (maximises $\dot\lambda$) appears to try to strengthen this almost invariance.

Although the questions posed in this paper are inspired by \cite{ADF} and \cite{AFG}, which considered related linear response optimisation for finite-state Markov chains and Hilbert-Schmidt operators, respectively, the deterministic case treated here required a  redevelopment of the optimisation approach, and an entirely distinct and more challenging perturbation theory. 
In particular, a certain amount of technical work was needed  to get explicit formulas for the response of the isolated spectrum,  see Appendix \ref{app}.
The question of optimising the outer spectrum (and therefore the mixing rate) has been addressed for flows in the presence of small noise, when the underlying vector field is periodically \cite{FS} and aperiodically \cite{FKS} driven.

Other related works consider the  problem of finding the   infinitesimal perturbation  achieving a prescribed \emph{desired} response direction (in the case of many infinitesimal perturbations achieving the same response one again looks for an optimal one). This problem was also called the ``linear request problem''  and studied in \cite{Kl,GP,GG} from a theoretical point for some classes of deterministic and random maps, and in applications in cellular automata \cite{MK} and climate \cite{LB}.
In particular the work \cite{GP} considers the case of expanding maps and the problem of finding a minimum-norm infinitesimal perturbation resulting in a given response for the invariant density of the system.

An outline of the paper is as follows.
In Section 2 we formalise the class of dynamical systems and perturbations we consider.
Proposition \ref{Prop:LRmain} summarises the fundamental theory concerning differentiability of invariant densities and spectral points for expanding maps.
Section 3 recaps relevant results from convex optimisation.
Section 4 formally sets up the optimisation problem to maximise the rate of change of the expectation of an observation function $c$. 
Proposition \ref{expect-resp-cts} states that there is a unique maximiser $\dot{T}$ and Theorem \ref{mainthm1} provides explicit expressions for the Fourier coefficients of the optimal $\dot{T}$.
Section 5 sets up the response maximisation problem for isolated spectral values $\lambda$.  Proposition \ref{Jcont} verifies there is a unique optimal perturbation $\dot{T}$ and Theorem \ref{mainthm2} states explicit formulae for the Fourier coefficients of the optimal $\dot{T}$ and the optimal value of the corresponding response of $\lambda$.
Section 6 describes our computational approach to estimate all of the relevant objects required to numerically solve the two main optimisation problems, and Section 7 illustrates our theory and numerics through two examples.
Finally, in the Appendix, we recall some known results about linear response of invariant measures and resolvent operators, and from these facts we derive general response formulas for invariant measures and isolated eigenvalues we use in the paper.


\section{Linear response of invariant densities and isolated eigenvalues \label{secmap}}

In this section we consider the response of the physical measure of an expanding circle map, and of the leading eigenvalues of the associated transfer operator, to deterministic perturbations of the map.
We begin by setting up the class of dynamical systems we will consider.
Proposition \ref{Prop:recall} then recalls known continuity properties for the physical measure and isolated spectrum under deterministic perturbations of the map, and Proposition \ref{Prop:LRmain} states the corresponding linear response results.
Proposition \ref{Prop:LRmain} also contains explicit formulas for the response of these objects under suitable deterministic perturbations of the system. 
These explicit formulas will be used in the computation of the optimal (response-maximising) perturbation.

Several linear response results in the literature treat perturbations that compose the dynamics with a diffeomorphism near to the identity, e.g.\ \cite{Ba2}.
In this paper we consider additive perturbations applied directly to the map, as we believe this is natural for applications.

Let us consider  some $\bar\delta>0$  and a family of $C^3$ maps  $\{T_{\delta}:S^{1}\rightarrow S^{1}\}_{\delta \in [0,\bar{\delta})}$
satisfying the following assumptions:
\begin{enumerate}
\item[(A0)] $T_0 \in C^4(S^1).$ 

\item[(A1)] There is $ \alpha<  1$ such that $|T_{0}^{\prime }(x)|\geq \alpha
^{-1}$ for all $x\in S^{1}$.

\item[(A2)]  The dependence of the family $T_\delta$ on $\delta $ is differentiable at $\delta=0$
in the following strong sense:
\begin{equation}\label{type}
T_{\delta }(x)=T_{0}(x)+\delta \dot{T} (x)+O_{C^{3}}(\delta^2)
\end{equation}%
where $\dot{T} \in C^{3}(S^{1},S^1)$ and we say
 a family of functions $F_\delta\in C^3(S^1,S^1)$ is $O_{C^{3}}(\delta^2)$
if
\begin{equation}\label{not492}
\underset{\delta
\in (0,\overline{\delta})}{\sup}\frac{||F_\delta||_{C^{3}}}{\delta ^{2}}<\infty.
\end{equation}

\end{enumerate}
We  study the statistical properties of these map
perturbations through their associated transfer operators. 
It is well known that if we consider the action of $L_{\delta }$ on a suitable Sobolev space, this operator is quasi-compact (see e.g.\ \cite{L2}).
We denote the Sobolev space of functions having weak $k$th derivatives in $L^{p}$ by $W^{k,p}$.
We recall
the definition of the transfer operator associated to a map of
the circle and of the derivative operator associated to perturbations as in (A2).
\begin{definition}\label{1}
The transfer operator $L_{{\delta }}:W^{1,1}(S^1,\mathbb{C})\rightarrow
W^{1,1}(S^1,\mathbb{C})$ associated to an expanding map $T_{\delta
}:S^{1}\rightarrow S^{1}$ is defined by
\begin{equation*}
L_{{\delta }}f(x)=\sum_{y\in T^{-1}_\delta x}\frac{f(y)}{T_{\delta
}^{\prime }(y)}
\end{equation*}%
for each $f\in W^{1,1}(S^{1},\mathbb{C}).$
The derivative operator $\dot{L}:W^{1,1}(S^1,\mathbb{C})\rightarrow L^1(S^1,\mathbb{C})$ associated
to a map perturbation in the direction $\dot{T}\in C^3(S^1,S^1)$ is defined as 
\begin{equation}
\dot{L}(f):=-L_0\left(\left[\frac{f \dot{T}}{T_0'}\right]'\right).
\label{der}
\end{equation}
\end{definition}
For the moment, we simply call $\dot L$ the ``derivative operator'';  in Appendix \ref{9.4.2} 
we show that $\dot L$ indeed arises as a derivative of the family of operators $L_\delta$ at $\delta=0$.
We denote by $f_\delta$ an invariant probability density for $T_{\delta }$.
Such densities are fixed points of the operators ${L}_{{\delta }} $;  that is, ${L}_{\delta }f_{\delta}=f_{\delta}$.
It is well known that an expanding map has a unique invariant probability density which is absolutely continuous with respect to the Lebesgue measure on $S^1$ (see e.g. \cite{KS,L2}).

We now recall a series of facts about stability  (Proposition \ref{Prop:recall}) 
 and linear response (Proposition \ref{Prop:LRmain}) of the invariant density   and isolated spectrum of expanding maps when subjected to deterministic perturbations.
The following proposition is a well-known fact about the stability of the spectral picture of exanding maps, which can be obtained by the results of
\cite{KL} (see also \cite{L2} for details).

\begin{proposition} \label{Prop:recall}
Consider a family of maps $T_\delta:S^1\to S^1$ for $\delta \in [0,\overline{\delta})$ satisfying (A0), (A1) and (A2).
Let $L_{\delta }:W^{1,1}(S^1,\mathbb{C})\rightarrow W^{1,1}(S^1,\mathbb{C})$ be the transfer operators associated  to $T_{\delta }$.
\begin{enumerate}
\item[(I)] There is a unique invariant probability density $f_\delta\in W^{2,1}(S^1,\mathbb{R})$ of $T_{\delta}$; 
 furthermore $||f_\delta-f_0||_{W^{1,1}}\to 0$ as $\delta \to 0$.
\item[(II)]
If $\lambda _{0}\in\mathbb{C}$ is  an
eigenvalue of $L_{0}$ acting on $W^{1,1}(S^1,\mathbb{C})$ such that $\alpha<|\lambda _{0}|<1$,  then $\lambda_0$ is isolated.
Furthermore, if $\lambda_0$ is simple, there is $\delta _{\ast }>0$ such that for  $\delta
\in [0,\delta _{\ast })$, $L_\delta$ has a simple isolated eigenvalue $\lambda_\delta$ and the map $\delta \mapsto \lambda _{\delta }$ is continuous.
\end{enumerate}
\end{proposition}


Linear response of the invariant measure of expanding maps under deterministic perturbations is a classical result (see \cite{Ba2}). A response formula for deterministic additive perturbations (see \eqref{type}) similar to \eqref{linr} was given in \cite{GP}.
Differentiability of the isolated eigenvalues of the transfer operator associated to expanding maps under deterministic perturbations is due to \cite{GL06} (see also \cite{Ba3}).
Since we need an explicit formula for this derivative, in Proposition \ref{Prop:LRmain} we also provide such a formula (see \eqref{lineig}).
To the best of our knowledge, the formula (\ref{lineig}) is new in the deterministic setting.
It mirrors the expression in Corollary III.11 \cite{HH}, which in our notation applies to a  family of quasi-compact operators $L_\delta$ that are continuously differentiable with respect to $\delta$ in operator norm.
The operator perturbations induced by a differentiable family $T_\delta$ of maps satisfying (A0)--(A2) do not fit into this framework and require a more careful treatment, this will be done using  theory developed in \cite{HH} and \cite{GL06}, as laid out in the Appendix.
\begin{proposition}
\label{Prop:LRmain}
Consider a family of maps $T_\delta:S^1\to S^1$ for $\delta \in (0,\overline{\delta})$  satisfying (A0), (A1) and (A2) as above.
\begin{enumerate}
\item[(I)] The following linear response formula for the invariant measure $f_\delta$ holds:
\begin{equation} \lim_{\delta \rightarrow 0}\left\Vert \frac{f_{\delta }-f_{0}}{\delta }+(Id-L_{0})^{-1}L_0\left(\left[\frac{f_0 \dot{T}}{T_0'}\right]'\right)\right\Vert _{L^{1}}=0.  \label{linr}
\end{equation}
\item[(II)]
Let $\lambda _{0}\in\mathbb{C}$ be  a simple, isolated eigenvalue of $L_{0}$ acting on $W^{1,1}(S^1,\mathbb{C})$ such that $\alpha<|\lambda _{0}|<1$.
Let $\varphi _{0}\in (W^{1,1}(S^1,\mathbb{C}))^*$ be the eigenfunction
of the
adjoint operator $L_{0}^{\ast }:(W^{1,1}(S^1,\mathbb{C}))^*\rightarrow (W^{1,1}(S^1,\mathbb{C}))^*$
associated to $\lambda _{0}$ i.e.\ $L_{0}^{\ast }\varphi _{0}=\lambda
_{0}\varphi _{0}$, where $\varphi _{0}$ is scaled so that $\varphi _{0}(%
\mathbf{1})=1$.
Let $v_{0}\in W^{1,1}(S^1,\mathbb{C})$ be the eigenfunction of $%
L_{0 }$ associated to $\lambda_{0 }$ scaled so that $\varphi
_{0}(v_{0 })=1$ (in fact $v_0\in W^{3,1}(S^1,\mathbb{C})$).
Furthermore, consider the continuous map $\delta \mapsto \lambda_{\delta }$ as in (II) of Proposition \ref{Prop:recall}. This map is  differentiable at $0$ and
\begin{equation}
\frac{d\lambda_{\delta }}{d\delta }|_{\delta =0}=\varphi _{0}(\dot{L}(
v_{0})) .  \label{lineig}
\end{equation}
\end{enumerate}
\end{proposition}


The proof of Proposition \ref{Prop:LRmain} is postponed to Appendix \ref{endofapp}.

\section{Recap of convex optimisation}
\label{opt}

We will consider a set $P\subset C^3(S^1,\mathbb{R})$ of allowed infinitesimal perturbations $\dot{T}$ to the map $T_0$.
We are interested in selecting an \emph{optimal} perturbation $\dot{T}$ in terms of (i) maximising the rate of change of the expectation of a chosen observable, and (ii) maximising the rate of change in the magnitude of an isolated eigenvalue.
Because we wish to perform an optimisation we consider $P$ inside some Hilbert space $\mathcal{H}$.
We will also assume that $P$ is bounded and convex, which we believe are natural hypotheses;  convexity because if two different
perturbations of a system are possible, then their convex combination --
applying the two perturbations with different intensities -- should also be
possible.
\begin{definition}
\label{stconv}We say that a convex closed set $P\subseteq \mathcal{H}$ is
\emph{strictly convex} if for each pair $x,y\in P$ and for all $0<\gamma<1$,
the points $\gamma x+(1-\gamma)y\in \mathrm{int}(P)$, where the relative
interior\footnote{%
The relative interior of a closed convex set $C$ is the interior of $C$
relative to the closed affine hull of $C$, see e.g.\ \cite{borwein}.} is
meant.
\end{definition}

We briefly recall some relevant results from convex optimisation.
Suppose $\mathcal{H}$ is a separable Hilbert space and $P\subset
\mathcal{H}$.
Let $\mathcal{J}:\mathcal{H}\rightarrow {\mathbb{R}}$ be a continuous
linear function.
Consider the abstract problem to find  $p_*\in P$ such that
\begin{equation}
\mathcal{J}(p_*)=\max_{p\in P}\mathcal{J}(p),  \label{gen-func-opt-prob}.
\end{equation}%
The existence and uniqueness of
an optimal perturbation follows from properties of $P$.
\begin{proposition}[Existence of the optimal solution]
\label{prop:exist} Let $P$ be bounded, convex, and closed in $\mathcal{H}$.
Then problem \eqref{gen-func-opt-prob} has at least one solution.
\end{proposition}
Upgrading convexity of the feasible set $P$ to strict convexity provides uniqueness of the optimum.
\begin{proposition}[Uniqueness of the optimal solution]
\label{prop:uniqe} Suppose $P$ is closed, bounded, and strictly convex
subset of $\mathcal{H}$, and that $P$ contains the zero vector in its
relative interior. If $\mathcal{J}$ is not uniformly vanishing on $P$
then the optimal solution to \eqref{gen-func-opt-prob} is unique.
\end{proposition}
Note that in the case $\mathcal{J}$ is uniformly vanishing, all the
elements of $P$ are solutions of the problem $(\ref{gen-func-opt-prob}).$
See Lemma 6.2 \cite{FKS}, Proposition 4.3 \cite{AFG}, or Corollary 3.23 of \cite{brezis} for the proofs of Propositions \ref{prop:exist} and \ref{prop:uniqe}
and more details.

\section{Optimization of the expectation of an observable.}
\label{subsubsec1}
Let $c\in L^{\infty}(S^{1},\mathbb{R})$ be an observable.
We consider the
problem of finding an infinitesimal perturbation $\dot{T}$ of our map $T_0$ that maximises the rate of change of expectation of $c$.
If $c$ were an indicator, for example, one could control the invariant density toward the support of $c$.
Given a
family of maps $T_{\delta }$ satisfying (A0), (A1), (A2)
with invariant densities $f_{\delta }$, we denote the response of the system
to $\dot{T}$ by
\begin{equation*}
R(\dot{T})=\lim_{\delta \rightarrow 0}\frac{f_{\delta }-f_{0}}{\delta }.
\end{equation*}%
This limit is converging in $L^{1}$ as proved in Proposition \ref{Prop:LRmain}. Under our assumptions we easily get
\begin{equation}
\lim_{\delta \rightarrow 0}\frac{\int \ c(x)f_\delta (x) \ dx-\int \ c(x)f_{0}(x) \ dx}{
\delta }=\int \ c(x)R(\dot{T})(x)\  dx. \label{LRidea2}
\end{equation}

Hence the rate of change of the expectation of $c$ with respect to $\delta$ is given by the linear
response of the system under the given perturbation.
To take advantage of the general results of the previous section, we perform the optimisation of $\dot{T}$ over a closed, bounded, convex subset of a suitable Hilbert space $\mathcal{H}$ containing the zero perturbation.
Because  we require $\dot{T}\in C^3$ in Proposition \ref{Prop:LRmain}
we select $\mathcal{H}=H^4(S^1,\mathbb{R})$ and consider a convex closed set $P\subseteq \mathcal{H}$.
To maximise the RHS of (\ref{LRidea2}) we set $\mathcal{J}(\dot{T}):=-\int c(x)\cdot R(\dot{T})(x)\ dx$ and set $P$ to be the unit ball in $H^4(S^1,\mathbb{R})$, which is bounded, convex, and contains the zero vector.
We hence consider the problem
\begin{eqnarray}
\label{optobs}
\min_{\dot{T}\in H^4}&&\mathcal{J}(\dot{T})\\
\label{optobs2}
\mbox{subject to}&&\|\dot{T}\|_{H^4}\le 1.
\end{eqnarray}
\begin{proposition}
\label{expect-resp-cts}
If $\mathcal{J}$ is not uniformly vanishing,  there is a unique optimum map perturbation $\dot{T}$ for Problem (\ref{optobs}).
\end{proposition}

\begin{proof}
The result will directly follow from Propositions \ref{prop:exist}  and \ref{prop:uniqe}, once we verify that $\mathcal{J}:H^4\to\mathbb{R}$ is continuous.
Since
\begin{equation*}
R(\dot{T} )=-(I-L_{0})^{-1}\left(L_{0}
\left[\left(
\frac{f_0\dot{T}}{T'_0}
\right)'\right]\right)
\end{equation*}%
we have $||R(\dot{T} )||_{L^1 }\leq ||R(\dot{T} )||_{W^{1,1}}\leq ||(I-L_0)^{-1}||_{V_1\to W^{1,1}}\cdot\|L_0\|_{W^{1,1}\to W^{1,1}}\cdot||\left(
\frac{f_0\dot{T}}{T'_0}
\right)'||_{W^{1,1}}$. The first two terms in the product are well known to be bounded in our case (see Proposition \ref{propora} and
Lemma \ref{Lemsu}). 
Because $f_0\in W^{3,1}$ and $T_0'\in C^3$ is uniformly bounded below, there exists $C_2$ such that $\left\|\left( \frac{f_0\dot{T}}{T'_0}
\right)'\right\|_{W^{1,1}} \leq C_2||\dot{T}||_{H^4}$
for some $C_2\geq 0$.
Thus, by linearity of $R(\dot{T})$, one sees that $\dot{T}\mapsto \int c(x)\cdot R(\dot{T})(x) \ dx$ is continuous.
\end{proof}

\subsection{An explicit formula for the optimal perturbation}
\label{sec:opt}


We now state a theorem identifying this optimal perturbation for the problem \eqref{optobs}.

\begin{theorem}
\label{mainthm1}
Let $c\in L^\infty(S^1,\mathbb{R})$ be an observation function and the family of maps $T_\delta$ satisfy (A0), (A1) and (A2).
The perturbation $\dot{T}\in H^4(S^1,\mathbb{R})$ that maximises the expected linear response $\int_{S^1} c(x)R(\dot T)(x)\ dx$ (solves Problem \eqref{optobs}--\eqref{optobs2}) is given by
$\dot{T}=\sum_{n=-\infty}^{\infty} a_ne_n$, where $e_n(x)=\exp(2\pi i nx)$ and
the coefficients $a_n$ are given by
\begin{equation}
    \label{aneqnthm}
a_n=\left[\int c\cdot \overline{(I-L_0)^{-1}L_0\left(\left[\frac{e_n f_0}{T'}\right]'\right)}\right]\left.\middle/ 2\nu\left(\sum_{k=0}^4 (4\pi^2n^2)^k\right)\right.,
\end{equation}
where $\nu>0$ is chosen so that $\|\dot{T}\|_{H^4}=1$;  that is $\sum_{n=-\infty}^{\infty}\sum_{k=0}^4 (4\pi^2n^2)^ka_n^2=1$.
\end{theorem}

\begin{proof}
We write the continuous (by the proof of Proposition \ref{expect-resp-cts}) linear functional $\mathcal{J}:H^4\to\mathbb{R}$ as an $L^2$ inner product
$$\mathcal{J}(\dot T)=\left\langle c,(I-L_0)^{-1}
L_{0}
\left(\left(
\frac{f_0\dot{T}}{T'_0}
\right)'\right)
\right\rangle,$$
Define a second continuous linear functional $g:H^4\to\mathbb{R}$ using the norm constraint: $g(\dot{T})=\|\dot{T}\|_{H^4}^2-1$.

For a general continuous linear functional $\mathcal{F}:H^4\to\mathbb{R}$ we denote by $\Delta\mathcal{F}(\dot{T},\tilde{T})$ the G\^ateaux variation of the functional $\mathcal{F}$ at $\dot{T}$ in the direction $\tilde{T}$, i.e.\ $\Delta\mathcal{F}(\dot{T},\tilde{T})=\lim_{h\to 0}{\mathcal{F}(\dot{T}+h\tilde{T})}/{h}.$
Given $\dot{T}\in H^4$, $\Delta\mathcal{J}(\dot{T},\tilde{T})$ exists for all $\tilde{T}\in H^4$ by linearity of $\mathcal{J}$;  indeed
\begin{equation}
    \label{fderivexpect}
\Delta\mathcal{J}(\dot{T},\tilde{T})=\mathcal{J}(\tilde{T})=\left\langle c,(I-L_0)^{-1}
L_{0}
\left(\left(
\frac{f_0\tilde{T}}{T'_0}
\right)'\right)
\right\rangle\qquad\mbox{ for all $\dot{T},\tilde{T}\in H^4$.}
\end{equation}
The variation of the functional $g$ is defined similarly, and it is straightforward to show that
\begin{equation}
    \label{gderivexpect}
\Delta g(\dot{T},\tilde{T})=2\langle \dot{T},\tilde{T}\rangle_{H^4}\qquad\mbox{ for all $\dot{T},\tilde{T}\in H^4$.}
\end{equation}
A variation $\Delta\mathcal{F}$ is called \emph{weakly continuous} if $\lim_{\dot{S}\to \dot{T}}\Delta\mathcal{F}(\dot{S},\tilde{T})=\Delta\mathcal{F}(\dot{T},\tilde{T})$ for each $\tilde{T}$.
It is clear from the explicit expressions given above that both $\Delta\mathcal{J}$ and $\Delta g$ are weakly continuous variations.

We form the Lagrangian $$\mathcal{L}(\dot{T},\nu)=\mathcal{J}(\dot{T})-\nu g(\dot{T}),$$
with Lagrange multiplier $\nu\in\mathbb{R}$.
Having established weak continuity of the variations of $\mathcal{J}$ and $g$, the Euler--Lagrange Multiplier Theorem \cite{smith} (section 3.3), guarantees that a necessary condition for $\dot{T}$ be a local extremum of the constrained Problem (\ref{optobs}) is that:



\begin{equation}
    \label{weakformexpect1}
\Delta\mathcal{L}(\dot{T},\tilde{T})=0\qquad\mbox{ for all $\tilde{T}\in H^4$.}
\end{equation}
and
\begin{equation*}
g(\dot{T})=0.
\end{equation*}
We express $\dot{T}\in H^4$ as $$\dot{T}=\sum_{m=-\infty}^{\infty} a_me_m,$$ with $a_m\in\mathbb{C}$.

By continuity of $\Delta\mathcal{J}(\dot{T},\cdot):H^4\to\mathbb{R}$ and density of smooth functions in $H^4$, we may equivalently insist that (\ref{weakformexpect1}) holds for each $\tilde{T}=e_n=\exp(2\pi i n x)$.
That is, for each $n\in\mathbb{Z}$ one has:


\begin{eqnarray*}
    0&=&\int c\cdot \overline{(I-L_0)^{-1}L_0\left(\left[\frac{e_n f_0}{T_0'}\right]'\right)}-2\nu\sum_{k=0}^4\int \left(\sum_m a_me_m\right)^{(k)}\cdot\overline{(e_n)^{(k)}}\\
    &=&\int c\cdot \overline{(I-L_0)^{-1}L_0\left(\left[\frac{e_n f_0}{T_0'}\right]'\right)}-2\nu\sum_{k=0}^4\int \left(\sum_m (2\pi i m)^ka_me_m\right)\cdot(-2\pi i n)^k\overline{(e_n)}\\
    &=&\int c\cdot \overline{(I-L_0)^{-1}L_0\left(\left[\frac{e_n f_0}{T_0'}\right]'\right)} -2\nu\left(\sum_{k=0}^4 (4\pi^2n^2)^k\right)a_n,
\end{eqnarray*}
where final equality uses $\int e_m\cdot \overline{e_n}=\delta_{m,n}$.
Thus, for $n\in\mathbb{Z}$, the coefficients $a_n$ are given by
\begin{equation}
    \label{aneqn}
a_n=\left[\int c\cdot \overline{(I-L_0)^{-1}L_0\left(\underbrace{\left[\frac{e_n f_0}{T_0'}\right]'}_{=:b_n}\right)}\right]\left.\middle/ 2\nu\left(\sum_{k=0}^4 (4\pi^2n^2)^k\right)\right.,
\end{equation}
where $\nu$ is chosen so that $\|\dot{T}\|_{H^4}=1$ to satisfy $g(\dot{T})=0$;  that is $\sum_{n=-\infty}^{\infty}\sum_{k=0}^4 (4\pi^2n^2)^ka_n^2=1$.
Notice that each $b_n$ has zero mean as the derivative of any periodic function integrates to zero.
Further note that the integral in (\ref{aneqn}) makes sense, because $e_n\in C^\infty$, $f_0\in C^3$ (since $T_0\in C^4$) 
and $T_0'\in C^3$, and therefore the function $b_n$ is in $C^2$ (in particular uniformly bounded with $|b_n|_\infty=O(n)$ due to the derivative acting on $e_n$) and in $W^{1,1}$.

We now verify that the above $a_n$ define a $\dot{T}\in H^4$.
Because we divide by $n^8$ in (\ref{aneqn}), the above growth estimate of $b_n$ leads to a decay rate of $|a_n|\le C/n^7$.
It is a fact that if the Fourier coefficients  $a_n(f)$ of some function $f$ satisfy $a_n(f)=O(1/n^{k+1+\gamma})$ for $\gamma>0$ then $f\in C^k$.
Thus the optimal perturbations $\dot{T}$ with Fourier coefficients given by $a_n$ in (\ref{aneqn}) lie in $C^5\subset H^4$.

Finally, we show that $\nu>0$.
Setting $\tilde{T}=\dot{T}$ in  (\ref{weakformexpect1}) and using (\ref{fderivexpect}) and (\ref{gderivexpect})
we have
$$0=\Delta\mathcal{J}(\dot{T},\dot{T})-\nu\cdot\Delta g(\dot{T},\dot{T})=\mathcal{J}(\dot{T})-2\nu\cdot\langle \dot{T},\dot{T}\rangle_{H^4}.$$
Because $\mathcal{J}(\dot{T})$ is maximal, we have $\mathcal{J}(\dot{T})>0$.
The above display equation then implies that $\nu>0$.

\end{proof}

\section{Optimisation of the spectrum}

We again consider a family of maps $\{T_\delta\}$  satisfying (A0), (A1) and (A2), and assume that $L_0$ has a simple eigenvalue $\lambda_0$ satisfying $|\lambda_0|>\alpha$.
By Proposition \ref{Prop:LRmain} the corresponding eigenfunction $v_0$ lies in $W^{3,1}$.
Proposition \ref{Prop:recall} guarantees the existence of a continuous family of simple eigenvalues $\lambda_\delta$ for the perturbed  operators $L_\delta:W^{1,1}\to W^{1,1}$.
Proposition \ref{Prop:LRmain} then show the differentiability of this family, providing a formula for its derivative at $\delta=0$, namely $\dot\lambda(\dot T)=\varphi_0(\dot L(\dot T)v_0)$.

We wish to optimise $\dot{\lambda}(\dot{T})$ as a function of the map perturbation $\dot{T}$.
This optimisation will be performed on the separable Hilbert space $H^{4}(S^1)\subset C^3(S^1)$.
We therefore define a linear functional $\mathcal{J}:H^4(S^1,\mathbb{R})\to\mathbb{C}$ as
\begin{equation}
    \label{Jdefn}
\mathcal{J}(\dot{T}):=\dot{\lambda}(\dot{T})=\varphi_0(\dot{L}(\dot{T})v_0).
\end{equation}
To simplify the explicit formulae appearing later in Section \ref{sec:explicit2} for the optimal perturbation $\dot T$, we assume that $\lambda_0$ is real and positive.
We therefore wish to select a perturbation $\dot{T}$ of $T_0$ so as to maximise the rate of increase of $\lambda_0$ under the perturbation $\dot{T}$ (in other words, maximise $\dot{\lambda}$):
\begin{eqnarray}
\label{specoptprob}
\max_{\dot{T}\in H^4}&& \mathcal{J}(\dot{T})\\
\label{specoptprob2}
\mbox{subject to}&& \|\dot{T}\|_{H^4}\le 1.
\end{eqnarray}

\begin{proposition}
\label{Jcont}
If $\mathcal{J}$ is not uniformly vanishing
there is a unique optimal map perturbation $\dot{T}$ for Problem (\ref{optobs}).
\end{proposition}
\begin{proof}
Because $H^4(S^1)$ is a separable Hilbert space and the  unit ball in $H^4(S^1)$ is a strictly convex, bounded, closed set containing the zero element, in order to apply Propositions \ref{prop:exist}  and \ref{prop:uniqe}, it is sufficient to check that $\mathcal{J}:H^4\to\mathbb{R}$ is continuous. We have to verify that for fixed $\varphi_0\in (W^{1,1})^*$ and $v_0\in W^{3,1}$ (see Proposition \ref{Prop:LRmain}), there is a $C<\infty$ such that $|\mathcal{J}(\dot{T})|\le C\|\dot{T}\|_{H^{4}}$.
We have $$|\mathcal{J}(\dot{T})|=|\varphi_0(\dot{L}(\dot{T})v_0)|\le \|\varphi_0\|_{(W^{1,1})^*}\|\dot{L}(\dot{T})v_0\|_{W^{1,1}},$$
and so it is sufficient to show that $\|\dot{L}(\dot{T})v_0\|_{W^{1,1}}\le C\|\dot{T}\|_{H^4}$.
We note that since $T_0\in C^4  $ with $T_0'>1$ and $v_0\in W^{3,1} $ we have $v_0/T_0'\in W^{3,1}$. Now
\begin{eqnarray*}\|\dot{L}(\dot{T})v_0\|_{W^{1,1}}&=&\|L_0((\dot{T}v_0/T_0')')\|_{W^{1,1}}\\
&\le& C\|(\dot{T}v_0/T_0')'\|_{W^{1,1}}\quad\mbox{by Lemma \ref{Lemsu}}.
\end{eqnarray*}
Thus we can conclude that there is a constant $\tilde{C}>0$ such that
\begin{equation*}
\|\dot{L}(\dot{T})v_0\|_{W^{1,1}}\le\tilde C\|\dot{T}\|_{H^4}.
\end{equation*}
\end{proof}

\subsection{Explicit formula for optimal solution}
\label{sec:explicit2}


We wish to minimise $\mathcal{J}:H^4(S^1,\mathbb{R})\to \mathbb{R}$, where $\mathcal{J}(\dot{T})=\varphi_0(\dot{L}(\dot{T})v_0)$, subject to $\|\dot{T}\|_{H^4}=1$.
From Proposition \ref{Jcont} we know there is a unique optimum.
In order to write an explicit formula for the optimal $\dot{T}$ we require a representative of $\varphi_0\in (W^{1,1}(S^1,\mathbb{R}))^*$ when acting on elements of $H^1(S^1,\mathbb{R})$.

\begin{lemma}
\label{duallemma}
Let $\varphi_0\in (W^{1,1}(S^1,\mathbb{R}))^*$. Then there is a $\phi_0\in H^1(S^1,\mathbb{R})$ such that
\begin{equation}
\label{dualrep}
\varphi_0(f)=\int_{S^1} \phi_0 \bar{f}\ dx + \int_{S^1} \phi_0' \bar{f'}\ dx \quad\mbox{
for all $f\in H^1(S^1,\mathbb{R})$.}
\end{equation}
\end{lemma}
\begin{proof}
Since $H^{1}(S^1)\subseteq W^{1,1}(S^1)$  we have that $(W^{1,1}(S^1))^*\subseteq (H^{1}(S^1))^*$ and therefore $\varphi_0\in (H^1(S^1))^*$.
The result follows by Riesz representation theorem.
\end{proof}
We now state a theorem identifying this optimal perturbation.
\begin{theorem}
\label{mainthm2}
Let the family of maps $T_\delta$ satisfy (A0), (A1) and (A2) and consider a family of isolated eigenvalues $\lambda_\delta$.
The perturbation $\dot{T}\in H^4(S^1,\mathbb{R})$ that maximises the expected linear response $\dot\lambda(\dot T)$ of  $\lambda_0$ (i.e.\ solves Problem (\ref{specoptprob})--(\ref{specoptprob2})) is given by
$\dot{T}=\sum_{n=-\infty}^{\infty} a_ne_n$, where $e_n(x)=\exp(2\pi i nx)$ and
the coefficients $a_n$ are given by
\begin{equation}
    \label{aneqnthm2}
a_n=\left[-\int \phi_0\cdot\overline{ L_0\left(\left(\frac{e_n v_0}{T'}\right)'\right)}-\int \phi_0'\cdot\overline{\left( L_0\left(\left(\frac{e_n v_0}{T'}\right)'\right)\right)'}\right]\left.\middle/ 2\nu\left(\sum_{k=0}^4 (4\pi^2n^2)^k\right)\right.,
\end{equation}
where $\nu>0$ is chosen so that $\|\dot{T}\|_{H^4}=1$;  that is $\sum_{n=-\infty}^{\infty}\sum_{k=0}^4 (4\pi^2n^2)^ka_n^2=1$.

Moreover, the maximal linear response is given by $$\dot\lambda(\dot T)=-\int \phi_0\cdot\overline{ L_0\left(\left(\frac{\dot T v_0}{T'}\right)'\right)}-\int \phi_0'\cdot\overline{\left( L_0\left(\left(\frac{\dot T v_0}{T'}\right)'\right)\right)'}.$$
\end{theorem}
\begin{proof}
We follow a similar strategy to the proof of Theorem \ref{mainthm1}.
Given $\dot{T}\in H^4$, we first need to show that  $\Delta\mathcal{J}(\dot{T},\tilde{T})=\lim_{h\to 0}(\mathcal{J}(\dot{T}+h\tilde{T})-\mathcal{J}(\dot{T}))/h$ exists for each $\tilde{T}\in H^4$.
Since $\tilde{T}\in C^3$ and $v_0\in C^3$, Proposition \ref{mainprop} yields $\dot{L}(\tilde T)v_0\in H^1$.
Thus by Lemma \ref{duallemma}, there is a $\phi_0\in H^1(S^1,\mathbb{R})$ such that
\begin{equation*}
\mathcal{J}(\tilde{T})=\varphi_0(\dot{L}(\tilde{T})v_0)=\int \phi_0\cdot \overline{\dot{L}(\tilde T)v_0} + \phi_0'\cdot \overline{(\dot{L}(\tilde T)v_0)'},
\end{equation*}
which is finite for each $\tilde{T}\in H^4$, and so by linearity of $\mathcal{J}$, we see that for each $\dot{T},\tilde{T}\in H^4$,
\begin{eqnarray*}
\Delta\mathcal{J}(\dot{T},\tilde{T})=\phi_0(\dot{L}(\tilde{T})v_0)&=&\int \phi_0\cdot \overline{\dot{L}(\tilde T)v_0} + \phi_0'\cdot \overline{(\dot{L}(\tilde T)v_0)'}\\
&=&-\int \phi_0\cdot \overline{L_0\left(\left(\frac{\tilde{T}v_0}{T'}\right)'\right)} - \phi_0'\cdot \overline{\left[L_0\left(\left(\frac{\tilde{T}v_0}{T'}\right)'\right)\right]'}.
\end{eqnarray*}
The variation of the functional $g$ is handled identically to the proof of Theorem \ref{mainthm1};  one obtains
\begin{equation*}
\Delta g(\dot{T},\tilde{T})=2\langle \dot{T},\tilde{T}\rangle_{H^4}\qquad\mbox{ for all $\dot{T},\tilde{T}\in H^4$.}
\end{equation*}
Weak continuity of the variations $\Delta\mathcal{J}$ and $\Delta{g}$ follows as in the proof of Theorem \ref{mainthm1}.




We form the Lagrangian
$$\mathcal{L}(\dot{T},\nu)=\mathcal{J}(\dot{T}) - \nu g(\dot{T}),$$
with Lagrange multiplier $\nu\in\mathbb{R}$.
Having established weak continuity of the variations of $\mathcal{J}$ and $g$, the Euler--Lagrange Multiplier Theorem \cite{smith} (section 3.3), guarantees that a necessary condition for $\dot{T}$ be a local extremum of the constrained Problem (\ref{specoptprob})--(\ref{specoptprob2}) is that:



\begin{equation}
    \label{weakformspec1}
\Delta\mathcal{L}(\dot{T},\tilde{T})=0\qquad\mbox{ for all $\tilde{T}\in H^4$.}
\end{equation}
and
\begin{equation*}
g(\dot{T})=0.
\end{equation*}
We express $\dot{T}=\sum_{m=-\infty}^{\infty} a_me_m,$ with $a_m\in\mathbb{C}$.
By continuity of $\Delta\mathcal{J}(\dot{T},\cdot):H^4\to\mathbb{C}$, and density of smooth functions  in $H^4$ we may equivalently insist that (\ref{weakformspec1}) holds for each $\tilde{T}=e_n=\exp(2\pi i n x)$.
That is, for each $n\in\mathbb{Z}$ one has:

\begin{eqnarray*}
    0&=&-\int \phi_0\cdot \overline{L_0\left(\left(\frac{e_nv_0}{T'}\right)'\right)}- \int \phi_0'\cdot\overline{\left( L_0\left(\left(\frac{e_nv_0}{T'}\right)'\right)\right)'}-2\nu\sum_{k=0}^4\int (\sum_m a_me_m)^{(k)}\cdot\overline{(e_n)^{(k)}}\\
    &=&-\int \phi_0\cdot\overline{ L_0\left(\left(\frac{e_nv_0}{T'}\right)'\right)}- \int \phi_0'\cdot\overline{\left( L_0\left(\left(\frac{e_nv_0}{T'}\right)'\right)\right)'} -2\nu\sum_{k=0}^4\int (\sum_m (2\pi i m)^ka_me_m)\cdot(-2\pi i n)^k\overline{(e_n)}\\
    &=&-\int \phi_0\cdot \overline{L_0\left(\left(\frac{e_nv_0}{T'}\right)'\right)}- \int \phi_0'\cdot\overline{\left( L_0\left(\left(\frac{e_nv_0}{T'}\right)'\right)\right)'}-2\nu\left(\sum_{k=0}^4 (4\pi^2n^2)^k\right)a_n,
\end{eqnarray*}
where final equality uses $\int e_m\cdot \overline{e_n}=\delta_{m,n}$.
Thus, for $n\in\mathbb{Z}$, the coefficients $a_n$ are given by
\begin{equation}
    \label{aneqn2}
a_n=\left[-\int \phi_0\cdot\overline{ L_0\left(\left(\frac{e_n v_0}{T'}\right)'\right)}-\int \phi_0'\cdot\overline{\left( L_0\left(\underbrace{\left(\frac{e_n v_0}{T'}\right)'}_{=:b_n}\right)\right)'}\right]\left.\middle/ 2\nu\left(\sum_{k=0}^4 (4\pi^2n^2)^k\right)\right.,
\end{equation}
where $\nu$ is chosen so that $\|\dot{T}\|_{H^4}=1$;  that is $\sum_{n=-\infty}^{\infty}\sum_{k=0}^4 (4\pi^2n^2)^ka_n^2=1$.
Notice that each $b_n$ has zero mean as the derivative of any periodic function integrates to zero.
The integrals in (\ref{aneqn2}) 
makes sense, because $e_n\in C^\infty$, $v_0\in C^3$ 
and $T'\in C^3$ and is bounded uniformly below by 1.
Therefore the function $b_n$ is in $C^1$ (in particular uniformly bounded with $|b|_\infty=O(n)$ due to the derivative acting on $e_n$) and in $H^1$.
We verify that the above $a_n$ define a $\dot{T}\in H^4$ in exactly the same way as in Theorem \ref{mainthm1}.
The fact that $\nu>0$ follows exactly as at the end of the proof of Theorem \ref{mainthm1}.
The final claim of the theorem follows from (\ref{Jdefn}), using Lemma \ref{duallemma} to represent $\varphi_0(\dot{L}(\dot T) v_0)$.
\end{proof}


\section{Numerical approach}

The computations revolve around evaluating the integrals in (\ref{aneqnthm}) and (\ref{aneqnthm2}).
We begin by discussing the common elements of these computations and then discuss specific elements in the subsequent subsections.
In this numerical section we will denote $L_0$ by simply $L$ to avoid confusion with the approximate transfer operator at resolution $N$, which we denote by $L_N$.

We first build a projection of $\hat{L}$ (the action of $L$ in frequency space) on complex exponentials $e_n(x)=\exp(2\pi i n x)$.
A finite spatial grid $\{0, 1/N, 2/N, \ldots, (N-1)/N\}\subset S^1$ corresponds to the $N$ Fourier modes $\{e_{-N/2+1},\ldots,e_{N/2}\}$ via discrete Fourier transform.
A matrix $\hat{L}_N$ is constructed from estimates of the integrals
$$\hat{L}_{N,nm}:=\int_{S^1} L(e_m)\cdot\bar{e}_n=\int_{S_1} e_m\cdot\bar{e}_n\circ T =\overline{\int_{S_1}\bar{e}_m\cdot e_n\circ T},$$
with the latter expression above estimated using fast Fourier transform of $e_n\circ T$ on a grid eight times finer than the $N$-grid.
Matrix multiplication by $\hat{L}_N$ updates the Fourier coefficients of a function, corresponding to applying $L$ to the function itself;  in particular $L(e_n)\approx\sum_{m} \hat{L}_{N,mn}e_m$.
See \cite{CF} for details.


\subsection{Numerical computation of the optimal response of the expectation of an observable}
\label{sec:expectnumerics}

Referring to (\ref{aneqnthm}),
\begin{enumerate}
    \item We estimate $f_0$ as the inverse transform of the leading eigenvector $\hat{f}_{N,0}$ of $\hat{L}_N$.
\item Using the spatial $N$-grid $0, 1/N, 2/N, \ldots, (N-1)/N$, we create $e_n f_0 / T'$ at these points by pointwise multiplication.
\item A circular central difference is taken to estimate $(e_n f_0/T')'$.
\item $L((e_n f_0/T')')$ is estimated in frequency space as $\hat{L}_N\left(\widehat{\left(e_n f_0/T'\right)'}\right)$.
\item To numerically estimate $(I-L)^{-1}L((e_n f_0/T')')$ we solve
the linear system $\hat{y}_N=(I-\hat{L}_N)\cdot \left(\hat{L}_N\left(\widehat{\left(e_n f_0/T'\right)'}\right)\right)$ and $\hat{y}_{N,1}=0$.
The latter condition ensures that the first (constant) Fourier mode is zero, corresponding to mean-zero functions, which is the appropriate subspace for the resolvent $(I-L)^{-1}$ to act on.
\item The integral with $c$ is performed by taking a dot product of the discrete Fourier transform of the observation $\hat{c}$ and $\hat{y}_N$.
\item The result is scaled by the appropriate denominator in (\ref{aneqnthm}) to produce $a_n$ up to a scaling factor controlled by $\nu$.
\end{enumerate}

In the examples we show results obtained by constraining the perturbation $\dot{T}$ according to the standard Sobolev norm $\|\cdot\|_{H^4}$ and weighted Sobolev norms  $\|f\|_{H^4,\gamma}^2=\sum_{i=0}^4 \int (f^{(i)}/\gamma^{i})^2,$ for $\gamma\gg 1$.
Increasing $\gamma$ penalises high derivatives less and allows for optimal $\dot{T}$ with greater irregularity.

\subsection{Numerical computation of optimising an isolated spectral value}
The isolated spectrum may be estimated by the outer spectrum of $\hat{L}$;  see \cite{CW} for formal statements.
We compute $\hat{L}_N$ as described above and consider one of its eigenvalues $\lambda_0$ satisfying $\lambda_0>1/\inf|T_0'|$.
Associated with $\lambda_0$ is an eigenvector $v_0$, which we estimate as the inverse Fourier transform of the eigenvector $\hat{v}_0$ of $\hat{L}_N$ corresponding to the eigenvalue $\lambda_0$.
Referring to (\ref{aneqnthm2}) we see that we must estimate $L((e_nv_0/T')')$ and the representative $\phi_0$  of $\varphi_0$ in $H^{1}$.
The former object is calculated according to steps 1--4 in Section \ref{sec:expectnumerics}, followed by an inverse FFT.
The calculation of the representative of $\varphi_0$ is described in the next subsection.

\subsubsection{Computing a representative $\phi_0$ of $\varphi_0$ in $H^{1}$}

By the adjoint eigenproperty of $\varphi_0$ we have \begin{equation}
    \label{adjointeqn}
\varphi_0(Lf)=\lambda_0\varphi_0(f)\qquad\mbox{ for all $f\in W^{1,1}$}.
\end{equation}
We seek a representative of $\varphi_0$ in $H^{1}$ and use the ansatz $\phi_0\approx \sum_{-N/2+1}^{N/2} a_me_m$.
We ask that (\ref{adjointeqn}) holds for $f=e_n$, $n=-N/2+1,\ldots,N/2$.
Thus, using (\ref{dualrep})  we wish to find $a_m, m=-N/2+1,\ldots,N/2$ such that
\begin{eqnarray}
\nonumber
\lefteqn{\int Le_n \cdot \overline{\sum_{m=-N/2+1}^{N/2} a_me_m} + \int (Le_n)'\cdot\overline{\left(\sum_{m=-N/2+1}^{N/2} a_me_m\right)'}}\\
\label{phi0H1eqn}&=& \lambda_0\left(\int e_n\cdot \overline{\sum_{m=-N/2+1}^{N/2} a_me_m} + \int (e_n)'\cdot\overline{\left(\sum_{m=-N/2+1}^{N/2} a_me_m\right)'}\right),
\end{eqnarray}
for $n=-N/2+1,\ldots,N/2$.



Let $L(e_n)\approx\sum_{p=-N/2+1}^{N/2} \hat{L}_{N,pn}e_p$.
With this approximation, writing out (\ref{phi0H1eqn}) more explicitly, we ask that
\begin{eqnarray*}
0&=&\int ({\lambda_0}-L)e_n\cdot\overline{\sum_{m=-N/2+1}^{N/2} a_m e_m} + \int [({\lambda_0}-L)e_n]'\cdot\overline{\left[\sum_{m=-N/2+1}^{N/2} a_m e_m\right]'}\\
&\approx&\int \left({\lambda_0}e_n-\sum_{p=-N/2+1}^{N/2} \hat{L}_{N,pn}e_p\right)\cdot\overline{\sum_{m=-N/2+1}^{N/2} a_m e_m}\\
\quad&&+ \int \left({\lambda_0}e_n'-\sum_{p=-N/2+1}^{N/2} (\hat{L}_{N,pn}e_p)'\right)\cdot\overline{\left[\sum_{m=-N/2+1}^{N/2} a_m e_m\right]'}\\
&=&\int \left({\lambda_0}e_n-\sum_{p=-N/2+1}^{N/2} \hat{L}_{N,pn}e_p\right)\cdot\overline{\sum_{m=-N/2+1}^{N/2} a_m e_m} \\
\quad&&+ \int \left({\lambda_0}(2\pi i n)e_n-\sum_{p=-N/2+1}^{N/2} \hat{L}_{N,pn}(2\pi i p)e_p\right)\cdot\overline{\sum_{m=-N/2+1}^{N/2} a_m (2\pi i m)e_m}\\
&=&{\lambda_0}\bar{a}_n-\sum_{p=-N/2+1}^{N/2} (\hat{L}_{N,pn}\bar{a}_p)
+  {\lambda_0}(2\pi i n)(-2\pi i n)\bar{a}_n - \sum_{n=-N/2+1}^{N/2} (\hat{L}_{N,pn}(2\pi i p)(-2\pi i p)\bar{a}_p) \\
&=&{\lambda_0}(1+4\pi^2n^2)\bar{a}_n
-\sum_{p=-N/2+1}^{N/2} \hat{L}_{N,pn}(1+4\pi p^2)\bar{a}_p.
\end{eqnarray*}
That is, ${a}$ satisfies
\begin{equation*}
\lambda_0\bar{a}_n=\sum_{p=-N/2+1}^{N/2} \frac{\hat{L}_{N,pn}(1+(2\pi p)^2)}{1+(2\pi n)^2}{\bar a_p}.
\end{equation*}
In other words, the conjugate $a$ coefficients form a left eigenvector of $\hat{L}_N$, suitably scaled.
Since $\lambda_0$ and the $\hat{L}_{N}$ are known numerically, we may easily solve for the $a_n$, $n=-N/2+1,\ldots,N/2$.

\section{Examples}

We illustrate Theorems \ref{mainthm1} and \ref{mainthm2} in the following two subsections.
In both cases we set $N=512$.

\subsection{Optimising the expectation of observables}

We define the expanding circle map $T(x)=2x-(0.9/2\pi)\sin(2\pi x)$ computed modulo 1.
The lower expansivity of $T$ at the fixed point $x=0$ leads to greater values of the invariant density nearby $x=0$.
A graph of $T$ and its invariant density are shown in Figure \ref{fig1}.

\begin{figure}
    \centering
    \includegraphics[scale=0.7]{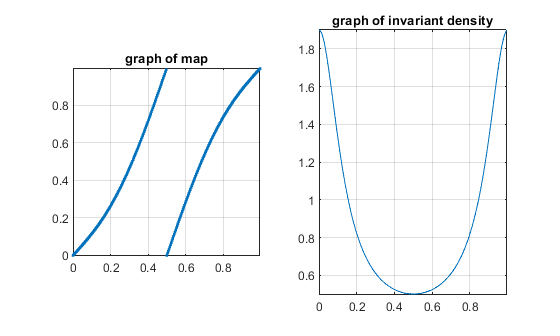}
    \caption{Graph of $x\mapsto 2x-(0.9/2\pi)\sin(2\pi x)$ and corresponding invariant density.}
    \label{fig1}
\end{figure}

Figure \ref{fig2} illustrates the optimal perturbations $\dot{T}$ for $c(x)=\cos(2\pi x)$ and various weights $\gamma$ used in the $H^4$ norm.
\begin{figure}
    \centering
    \includegraphics[width=.9\textwidth]{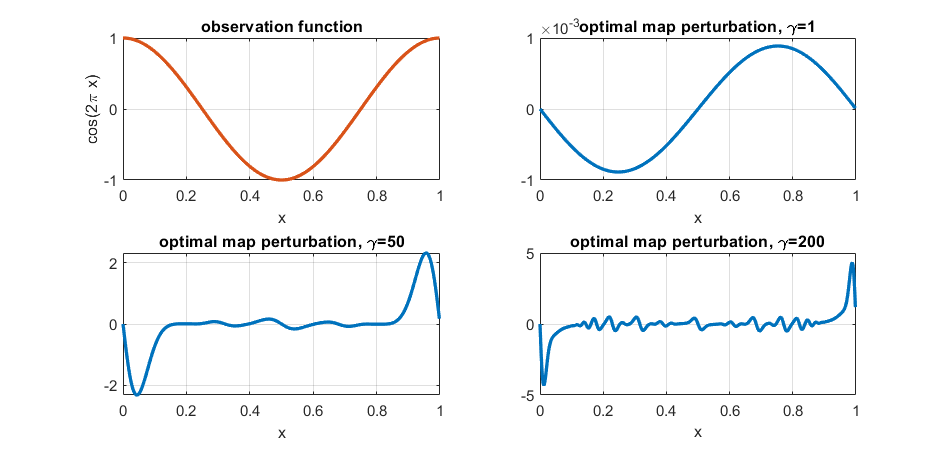}
    \caption{Upper left: observation function $c(x)=\cos(2\pi x)$. Other panels:  optimal map perturbations for different weights $\gamma$. Each perturbation has unit norm in its corresponding $\gamma$-weighted norm. As $\gamma$ increases, the optimal map perturbation may become more irregular as a smaller penalty is paid for the first to fourth-order derivatives.}
    \label{fig2}
\end{figure}
The map perturbations seek to increase the expectation of $c$. Because the maximal value of $c$ occurs at $x=0$ it is advantageous for the perturbations to retain the fixed point at $x=0$ while simultaneously reducing the expansivity of the map at the fixed point.
Such perturbations make the fixed point more ``sticky'' and will lead to invariant densities with even greater values at $x=0$, leading to increases in expectation of $c$.

Figure \ref{fig3} carries out the same experiments, replacing the observation function with $c(x)=\sin(2\pi x)$.
\begin{figure}
    \centering
    \includegraphics[width=.9\textwidth]{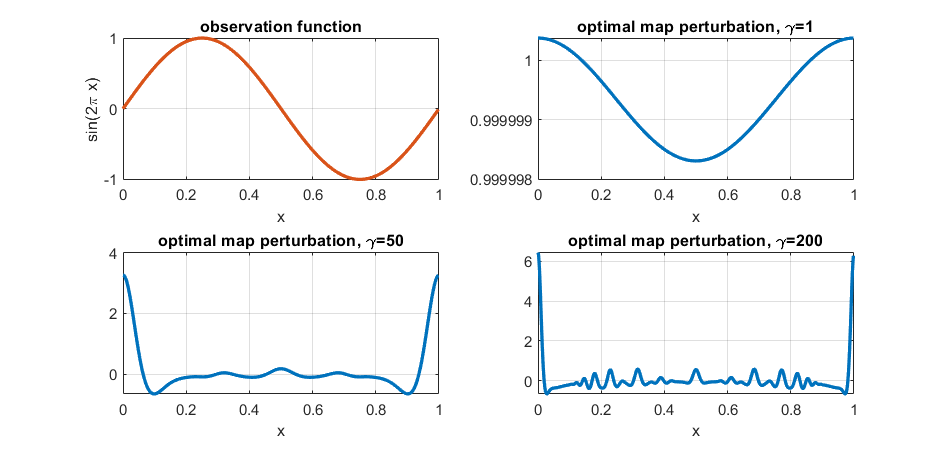}
    \caption{Upper left: observation function $c(x)=\sin(2\pi x)$. Other panels:  optimal map perturbations for different weights $\gamma$. Each perturbation has unit norm in its corresponding $\gamma$-weighted norm. As $\gamma$ increases, the optimal map perturbation may become more irregular as a smaller penalty is paid for the first to fourth-order derivatives.}
    \label{fig3}
\end{figure}
Now, there is an imperative to remove the sticky fixed point to move invariant mass away from $x=0$ and toward $x=0.25$, where $c$ takes its maximum.
As shown in Figure \ref{fig3}, the strategy is to displace the fixed point by moving it to the right from $x=0$.

\subsection{Optimising the spectrum}
\label{sec:7.2}
To optimise the spectrum, we require an isolated eigenvalue $\lambda_0$ of the transfer operator $L$ satisfying $|\lambda_0|>1/\inf|T_0'|$.
We construct a piecewise-linear Markov map of $S^1$ by linearly connecting the points (rounded to 4 decimal places) $$x\in\{ 0.0,
 0.1197,
 0.2045,
 0.2453,
 0.3369,
 0.3874,
 0.49,
 0.5875,
 0.6336,
 0.7343,
 0.7695,
 0.8523,
 1.0\}$$
 to their respective images (to be taken mod 1) $$T_0(x)\in\{0.0,
 0.1976,
 0.3032,
 0.4505,
 0.5885,
 0.7312,
 1.0,
 1.1976,
 1.3032,
 1.4505,
 1.5885,
 1.7312,
 2.0\}$$ as shown in Figure \ref{fig:gapmap} (left).
 \begin{figure}
 \begin{center}
\includegraphics[width=.49\textwidth]{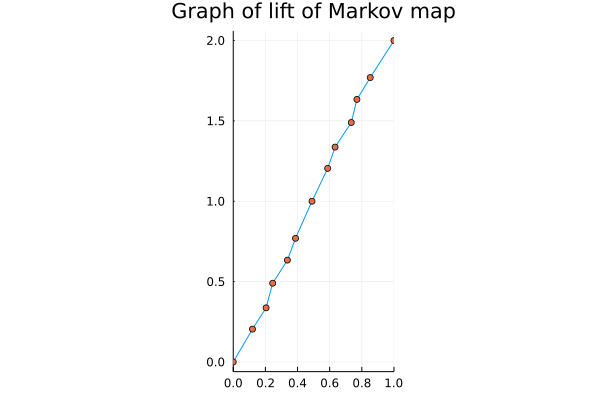}
\includegraphics[width=.49\textwidth]{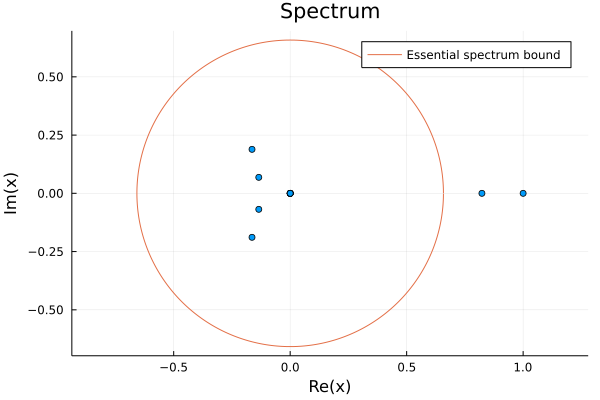}
\end{center}
\caption{Left: the graph of the lift of the two-full-branch piecewise-linear circle map constructed by joining the 13 point pairs $(x,T_0(x))$ listed toward the beginning of section \ref{sec:7.2} and shown as orange dots in the figure. Right: 12 eigenvalues, shown as blue dots, of the $12\times 12$ matrix representing the restriction of the transfer operator of the Markov map to the invariant subspace spanned by 12 indicator functions supported on the domains of linearity of the map. The essential spectrum bound produced by the inverse of the minimum slope of the 12 linear branches is shown in red.}
\label{fig:gapmap}
\end{figure}
\begin{figure}
    \centering
\includegraphics[width=0.65\textwidth]{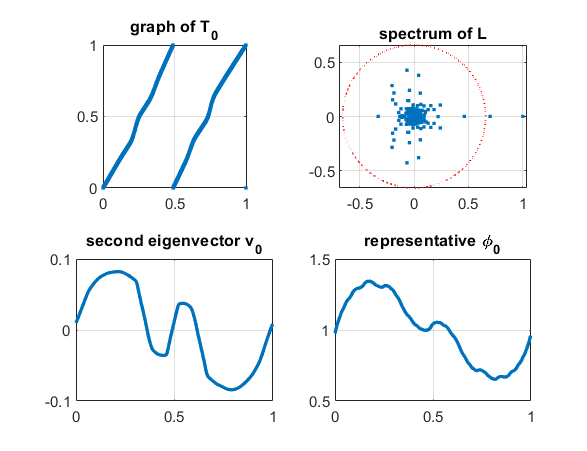}
    \caption{Upper left: graph of a smoothed $T_0$ derived from the map in Figure \ref{fig:gapmap} (left).  Upper right:  spectrum of $L_N$, with $1/\inf|T'_0|$ indicated as a dotted red circle.  Lower left: second eigenvector $u_0$, normalised so that $\varphi_0(u_0)=1$.  Lower right: representative of $\varphi_0$ in $W^{1,\infty}$, normalised so that $\varphi_0(\mathbf{1})=1$.}
    \label{fig:isolspecsummary}
\end{figure}
This Markov map has an isolated spectral value of  $\lambda_0\approx 0.8231$, while $1/\inf|T'_0|\approx 0.6579$;  see Figure \ref{fig:gapmap} (right).
We are not aware of another two-branch\footnote{Keller and Rugh \cite{KR} construct a two-branch circle map with an isolated negative eigenvalue;  by considering two iterates of this map, one would obtain a four-branch circle map with a positive isolated eigenvalue.} circle map in the literature whose Perron-Frobenius operator has a positive isolated eigenvalue strictly inside the unit circle, but larger than the reciprocal of the magnitude of minimal slope.  

We then smooth this map by convolving with a bump function $$b(x)=\exp\left(\frac{-1}{1-(x/\epsilon)^2}\right)\bigg/\kappa\epsilon,$$ using $\epsilon=1/40$;  $\kappa$ is a constant chosen so that $\int_{S_1} b =1$.
The smoothed map satisfies $1/\inf|T'_0|\approx 0.6579$ and the second largest magnitude eigenvalue of $\hat{L}_N$ for this smoothed map $T_0$ is $\lambda_0\approx 0.6992$;  as we no longer discuss the original piecewise linear map we reuse the notation $T_0$ and $\lambda_0$.
The graph of the smoothed map $T_0$, its numerical spectrum, and estimates of the eigenvector $v_0$ and the representative $\varphi_0$ are displayed in Figure \ref{fig:isolspecsummary}.

Figure \ref{fig:spectrum_optim} illustrates the optimal perturbations $\dot{T}$ to maximally increase the isolated spectral value.
\begin{figure}
\centering
\includegraphics[width=0.9\textwidth]{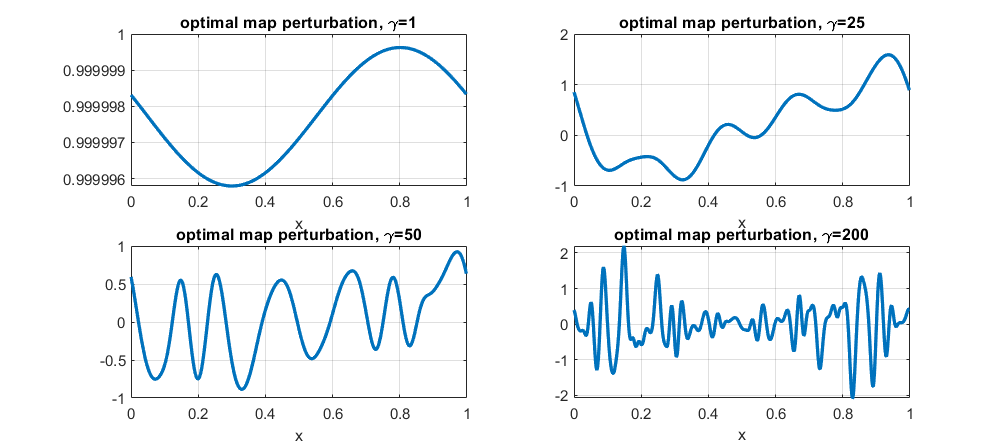}    \caption{Optimal perturbations $\dot{T}$ for different weights $\gamma$.  Each perturbation $\dot{T}$ has unit norm in its corresponding $\gamma$-weighted norm. As $\gamma$ increases, the optimal map perturbation may become more irregular because a smaller penalty is paid for the first to fourth-order derivatives. The values of $\dot\lambda(\dot T)$ are (in order upper left, upper right, lower left, lower right): 0.5758, 3.1427, 11.1590, and 125.51, demonstrating that as we allow greater irregularity in the $\gamma$-weighted norm we are able to increase the spectral response.}
\label{fig:spectrum_optim}
\end{figure}
Using Theorem \ref{mainthm2} we may also compute the linear responses of $\lambda_0$ with respect to the optimal map perturbations $\dot{T}$ shown in Figure \ref{fig:spectrum_optim} for various $\gamma$-weighted $H^4$ norms.
For $\gamma=1, 25, 50,$ and $200$ we find that $\dot{\lambda}(\dot T)$ is approximately $0.5758, 3.1427, 11.1590,$ and $125.51$, respectively.
These values indicate that as we reduce the penalty on the irregularity of the perturbations $\dot T$ by increasing $\gamma$, we may increase the corresponding linear response of $\lambda_0$ without limit.
To put these numbers in perspective, we remark that even a movement of the spectral value $\lambda_0$ by an amount 0.1 would be dramatic, and with $\gamma=200$, making a macroscopic peturbation of $T_0$ by $\dot{T}/1000$ would exceed such a movement (up to linear approximation).


\appendix
\section{General Linear Response formulas for eigenvalues and eigenvectors and application to  expanding maps\label{app}}

In this section we recall general results for the linear response of fixed
points, eigenvalues and eigenvectors of Markov operators under suitable
perturbations.
We then develop the estimates that are necessary to apply the
results to expanding maps and deterministic perturbations. Let $X$ be a
compact Riemann manifold and let $m$ be its normalized volume measure. Denote by $L^1(X,m)$ or simply by $L^1$ the space of $m$-integrable functions.
We consider a sequence of Banach
spaces
$$ (B_{ss},\Vert ~\Vert _{ss})\subseteq (B_{s},\Vert ~\Vert _{s})\subseteq
(B_{w},\Vert ~\Vert _{w}),$$ with $B_w\subseteq L^1(X,m)$ and the norms satisfying
\begin{equation*}
\Vert ~\Vert _{w}\leq \Vert ~\Vert _{s}\leq \Vert ~\Vert _{ss}.
\end{equation*}
Let us suppose that $B_{ss}$ contains the constant functions.
For $i\in \{ss,s,w\}$ define the following (closed) zero-mean function spaces $V_{ss}\subseteq
V_{s}\subseteq V_{w}$ by $V_{i}:=\{f \in B_{i}:\int_X f\ dm=0\}.$
We will consider Markov\footnote{%
A Markov operator $L:B\to B$ is a linear operator satisfying (i) $L(f)\ge 0$ if $f\ge 0$ and (ii)
$\int_X L(f)\ dm=\int_X f\ dm$ for all $f\in B$.} 
operators
acting on these spaces.
If $A,B$ are two
normed vector spaces and $T:A\rightarrow B$ we denote the mixed norm $\Vert
T\Vert _{A\rightarrow B}$ as
\begin{equation*}
\Vert T\Vert _{A\rightarrow B}:=\sup_{f\in A,\Vert f\Vert _{A}\leq 1}\Vert
Tf\Vert _{B}.
\end{equation*}

\subsection{Abstract linear response of invariant densities}

Under the general assumptions and notations above we are now going to state an abstract result on the linear response of invariant densities. Similar results and constructions appear in the literature, applied to specific classes of examples, we include a proof of the statement for completeness.

\begin{theorem}
\label{LR}
Let us consider  $\overline{\delta}>0$, $\delta \in [ 0,\overline{\delta})$ and a family of Markov operators $L_\delta:B_w\to B_w.$ 
Suppose that for $\delta \in \lbrack 0,\bar{\delta})$ there is a
probability density $\ v_{\delta }\in B_{s}$ such that 

\begin{equation*}
L_{\delta }v_{\delta }=v_{\delta }
\end{equation*}%
and that there is $\dot{L}v_{0}\in B_{s}$ such that%
\begin{equation}\label{resex}
\lim_{\delta \rightarrow 0}\left\|\frac{L_{\delta }-L_{0}}{\delta }v_{0}-\dot{L}%
v_{0}\right\|_{s}=0.
\end{equation}%
Suppose that for  $\delta \in [ 0,\overline{\delta})$ the resolvent operators $(Id-L_\delta)^{-1}$ 
of $L_\delta$ are defined and bounded from $V_{s}$ to $
V_{w},$

\begin{equation}\label{resbound}
||(Id-L_{\delta})^{-1}||_{V_{s}\rightarrow V_{w}}<+\infty \end{equation}
and
\begin{equation}\label{reslim}
\lim_{\delta \rightarrow 0}||(Id-L_{\delta
})^{-1}-(Id-L_{0})^{-1}||_{V_{s}\rightarrow V_{w}}=0.
\end{equation}%
Then
\begin{equation*}
\lim_{\delta \rightarrow 0}\left\|\frac{v_{\delta }-v_{0}}{\delta }%
-(Id-L_{0})^{-1}\dot{L}v_{0}\right\|_{V_{w}}=0.
\end{equation*}
\end{theorem}

\begin{proof}
We have that for each $\delta \in \lbrack 0,\bar{\delta})$, $v_{\delta }$ is
a fixed point of $L_{\delta }.$ Using this we get%
\begin{eqnarray*}
(Id-L_{\delta })\frac{v_{\delta }-v_{0}}{\delta } &=&\frac{v_{\delta }-v_{0}%
}{\delta }-\frac{L_{\delta }v_{\delta }-L_{\delta }v_{0}}{\delta } \\
&=&\frac{-v_{0}+L_{\delta }v_{0}}{\delta } \\
&=&\frac{1}{\delta }(L_{\delta }-L_{0})v_{0}.
\end{eqnarray*}%
We remark that for each $\delta ,$ $L_{\delta }$ preserves $V_{s}$. Since $%
\forall \delta >0$, $\frac{L_{\delta }-L_{0}}{\delta }v_{0}\in V_{s}$ \ and by the assumptions, \eqref{resbound}, \eqref{reslim},
for $\delta $ small enough
$(Id-L_{\delta })^{-1}:V_{s}\rightarrow V_{w}$ is a uniformly bounded operator. We can
apply the resolvent to both sides of the expression above to get%
\begin{eqnarray}
(Id-L_{\delta })^{-1}(Id-L_{\delta })\frac{v_{\delta }-v_{0}}{\delta }
&=&(Id-L_{\delta })^{-1}\frac{L_{\delta }-L_{0}}{\delta }v_{0}  \label{xx1}
\\
&=&(Id-L_{\delta })^{-1}\frac{L_{\delta }-L_{0}}{\delta }%
v_{0}-(Id-L_{0})^{-1}\frac{L_{\delta }-L_{0}}{\delta }v_{0}  \notag \\
&&+(Id-L_{0})^{-1}\frac{L_{\delta }-L_{0}}{\delta }v_{0}.
\end{eqnarray}%
Since $||(Id-L_{\delta })^{-1}-(Id-L_{0})^{-1}||_{V_{s}\rightarrow
V_{w}}\rightarrow 0$ \ we have%
\begin{eqnarray*}
||[(Id-L_{\delta })^{-1}-(Id-L_{0})^{-1}]\frac{L_{\delta }-L_{0}}{\delta }%
v_{0}||_{w} &\leq &||(Id-L_{\delta
})^{-1}-(Id-L_{0})^{-1}||_{V_{s}\rightarrow V_{w}}||\frac{L_{\delta }-L_{0}}{%
\delta }v_{0}||_{s} \\
&\rightarrow &0.
\end{eqnarray*}%
Since $\lim_{\delta \rightarrow 0}\frac{L_{\delta }-L_{0}}{\delta }v_{0}$
converges in $V_{s},$ then $\ (\ref{xx1})$ \ implies that in the $B_{w}$
topology%
\begin{eqnarray*}
\lim_{\delta \rightarrow 0}\frac{v_{\delta }-v_{0}}{\delta } &=&\lim_{\delta
\rightarrow 0}~(Id-L_{0})^{-1}\frac{L_{\delta }-L_{0}}{\delta }v_{0} \\
&=&(Id-L_{0})^{-1}\lim_{\delta \rightarrow 0}\frac{L_{\delta }-L_{0}}{\delta
}v_{0} \\
&=&(Id-L_{0})^{-1}[\dot{L}v_{0}].
\end{eqnarray*}
\end{proof}

\subsection{Abstract linear response of the resolvent and linear response of
eigenvalues\label{sect:GL}}

In this section we provide general statements about the response of eigenvalues and resolvent operators.
The results presented follow from classical statements we adapt to our purposes.
Corollary \ref{C15} illustrates the abstract linear response of the resolvent and Proposition \ref{eigdiff} provides differentiability of an isolated eigenvalue and of its corresponding eigenvector in $B_w$.

We recall the abstract  linear response result for the resolvent operator
proved in \cite{GL06} and we apply it to get linear response formulas
for simple eigenvalues and eigenvectors of transfer operators.
Recalling the Banach spaces $B_{ss}, B_s, B_w$ from the last subsection, 
let us suppose that all are separable and that $B_{ss}$ is dense in $ B_{s}$.
Let us consider  $\overline{\delta}>0$, $\delta \in [ 0,\overline{\delta})$ and a family of Markov operators $L_\delta$ and an operator $\dot{L}:B_s \to B_w$ satisfying the following assumptions:
there are $C\ge 0$ and $0\le \alpha<1$ such that for each $n\in \mathbb{N}$ and $\delta \in \lbrack 0,\overline{\delta})$:

\begin{enumerate}
\item[(GL1)] $||L_{\delta }^{n}||_{w}\leq C||f||_{w}$

\item[(GL2)]\label{newLY}
 $||L_{\delta }^{n}||_{s}\leq C\alpha
^{n}||f||_{s}+C||f||_{w}$, $||L_{\delta }^{n}||_{ss}\leq C\alpha
^{n}||f||_{ss}+C||f||_{s}$

\item[(GL3)]  $||L_{\delta
}||_{B_{s}\rightarrow B_{s}}\leq C$, and $||L_{\delta }||_{B_{ss}\rightarrow
B_{ss}}\leq C$

\item[(GL4)] $||L_{\delta }-L_{0}||_{B_{s}\rightarrow B_{w}}\leq C\delta $
and $||L_{\delta }-L_{0}||_{B_{ss}\rightarrow B_{s}}\leq C\delta .$

\item[(GL5)]  $||\dot{L}||_{B_{s}\rightarrow B_{w}}\leq C$,  $||\dot{L}%
||_{B_{ss}\rightarrow B_{s}}\leq C$ and

\item[(GL6)] $||L_{\delta }-L_{0}-\delta \dot{L}||_{B_{ss}\rightarrow
B_{w}}\leq C\delta ^{2}$.
\end{enumerate}
Under these assumptions one has:
\begin{theorem}[\cite{GL06}]
\label{GL} Consider a family of operators $L_\delta$ satisfying the assumptions (GL1),...,(GL6).
Further consider $\delta'>0 $, $\alpha'>\alpha$ and the set $$Q_{\delta',\alpha'}=\{z\in \mathbb{C}: |z|\geq\alpha', \mbox{dist}(\sigma(L_0),z)>\delta'\}$$ where $\sigma(L_0)=\sigma_s(L_0)\cup \sigma_{ss}(L_0) $ and $\sigma_{s}(L_0), \sigma_{ss}(L_0)$ denote the spectrum of $L_0$ acting on $B_{s}$ and $B_{ss}$ respectively.

Then there is $C_2\geq 0$ such that for each $z\in Q_{\delta',\alpha'}$ and $\delta$ small enough
\begin{equation}
||(z-L_{\delta })^{-1}-(z-L_{0})^{-1}(Id+\delta \dot{L}%
(z-L_{0})^{-1})||_{B_{ss}\rightarrow B_{w}}\leq C_2|\delta |^{1+\eta }
\label{ww}
\end{equation}%
where $\eta =\log (\alpha' /\alpha )/\log (1/\alpha )$.
\end{theorem}

We remark that $(\ref{ww})$ gives the first-order change of the resolvent
when the operator is perturbed; in fact from (\ref{ww})
one has the following immediate rearrangement%
\begin{equation}
||(z-L_{\delta })^{-1}-(z-L_{0})^{-1}+\delta (z-L_{0})^{-1}\dot{L}%
(z-L_{0})^{-1})||_{B_{ss}\rightarrow B_{w}}\leq C_2|\delta |^{1+\eta }
\label{1112}
\end{equation}%
and the following Corollary:

\begin{corollary}
\label{C15}
Under the above assumptions and with the same notations
\begin{equation*}
\left\|\frac{(z-L_{\delta })^{-1}-(z-L_{0})^{-1}}{\delta }-(z-L_{0})^{-1}\dot{L}%
(z-L_{0})^{-1}\right\|_{B_{ss}\rightarrow B_{w}}\leq C_2|\delta |^{\eta }.
\end{equation*}
\end{corollary}

We will also make the following assumption on the spaces involved, which is well known to imply together with the  assumptions (GL1) and (GL2), the quasicompactness of the transfer operators $L_\delta$ when acting on $B_s$ and on $B_{ss}$.
\begin{enumerate}
    \item[(GL7)]  $B_{s}$ is compactly
immersed in $B_{w}$ and $B_{ss}$ is compactly immersed in $B_{s}$.
\end{enumerate}

Let $\mathbf{1}$ be the  density representing the indicator function of $S^1$.
Let $\lambda _{0}$ be a simple isolated eigenvalue for $L_{0}$ acting on $B_s$.  Lemma III.3 of \cite{HH} now ensures the existence of a unique eigenfuction   $\varphi _{0}\in B_{s}^{\ast }$  of the adjoint operator $L_{0}^{\ast }:B_{s}^{\ast }\rightarrow B_{s}^{\ast }$ corresponding to  $\lambda _{0}$ $(L_{0}^{\ast }\varphi _{0}=\lambda _{0}\varphi
_{0})$, scaled so that $\varphi _{0}(\mathbf{1})=1$.
Suppose $\lambda _{\delta}$ is an isolated, simple eigenvalue for $L_{\delta}$ acting on $B_s$. Suppose $v_{\delta }$ is an eigenvector   of $L_{\delta }$  associated to $\lambda _{\delta }$.
To quantitatively address the differentiability of $v_{\delta }$ and $%
\lambda _{\delta }$ we need to scale $v_{\delta }$ consistently.
  We will rescale $v_{\delta }$ in a way
that $\varphi _{0}(v_{\delta })=1$ for all $\delta \in \lbrack 0,\overline{%
\delta}]$.
Let us consider a simple eigenvalue
$\lambda_\delta \in \sigma (L_\delta),$ and $\theta >0$ \
such that $$\{|z-\lambda_\delta |\leq \theta \}\cap \sigma (L_\delta)=\{\lambda_\delta \}.$$ The
eigenprojection operator
$\Pi _{\lambda_\delta}:B_{s}\rightarrow B_{s}$ is defined by%
\begin{equation}\label{project3}
\Pi _{\lambda_\delta }:=\frac{1}{2\pi i}\int_{\{|z-\lambda_\delta |=\theta
\}}(z\cdot Id-L_{\delta })^{-1}dz.
\end{equation}
 It is well known that this is a projection ($\Pi _{\lambda_\delta }^{2}=\Pi
_{\lambda_\delta }$) and it does not depend on $\theta$ or on the circle $\{|z-\lambda_\delta |=\theta \}$, which can  be replaced by any smooth simple curve only containing the simple eigenvalue $\lambda_\delta$.
Furthermore $L_{\delta}=\Pi
_{\lambda_\delta }L_{\delta}\Pi _{\lambda_\delta }+(Id-\Pi _{\lambda_\delta})L_{\delta}(Id-\Pi _{\lambda_\delta}) $ and $\sigma (\Pi _{\lambda_\delta }L_{\delta}\Pi _{\lambda_\delta })=\{0,\lambda_\delta \}$.
Hence $\Pi _{\lambda_\delta}$ is a projection on the $\lambda_\delta $-eigenspace of $%
L_{\delta}$.

 We now consider the dependence of the isolated eigenvalue $\lambda_\delta$ on $\delta$.  First we consider the associated eigenvector $v_\delta$ and state a formula for its derivative with respect to $\delta$.

\begin{proposition}
\label{eigdiff}Suppose the family of transfer operators $L_{\delta }$
satisfy the assumptions (GL1),...,(GL7). Suppose $|\lambda _{0}|>\alpha$ (see (GL2)) is a simple isolated eigenvalue for $L_0$ acting on $B_{s}$. 
Then
\begin{enumerate}
    \item[(I)] $\lambda_0$ is also an eigenvalue of the operator applied to $B_{ss}$. Furthermore, for $\delta $  small enough $L_\delta$ has a family of simple isolated  eigenvalues (both for the operator applied to $B_s$ and $B_{ss}$) $\lambda_\delta$ with $\lambda_\delta\to \lambda_0$ as $\delta \to 0$.
    \item[(II)]  each $\lambda_\delta$ has an eigenvector $v_{\delta}\in B_{s}$, rescaled by $\varphi
_{0}(v_{\delta })=1$ as described before, and as $\delta \rightarrow 0$%
\begin{equation}
\lim_{\delta \rightarrow 0}||v_{0}-v_{\delta }||_{s}=0. \label{sstab}
\end{equation}
Further, one has 
\begin{equation}
\frac{v_{\delta }-v_{0}}{\delta }\rightarrow \dot{v}:=\frac{1}{2\pi i}%
\int_{\{|z-\lambda _{0}|=\theta \}}(z-L_{0})^{-1}\dot{L}(z-L_{0})^{-1}v_{0}dz
\end{equation}%
with convergence in $B_{w}$.
Moreover, the function $\delta \rightarrow \lambda _{\delta }$ is differentiable.
\end{enumerate}
\end{proposition}

\begin{proof}
Lemma A.3 of \cite{Ba3} (following a similar result proved in \cite{BT})
implies that if (i) $B_s$ is separable and $L:B_s\rightarrow B_s$ is a continuous
linear map preserving a dense continuously embedded subspace $B_{ss}$, (ii) $L:B_{ss}\rightarrow
B_{ss}$ is continuous and (iii) the essential spectral radius of $L$
considered both as acting on $B_s$ and on $B_{ss}$ is
bounded by $0<\rho <1$, then the simple eigenvalues of $L:B_s
\rightarrow B_s$ and $L:B_{ss}\rightarrow B_{ss}$
in $\{z\in \mathbb{C}|~||z||>\rho \}$ coincide, the associated
eigenspaces also coincide and are contained in $B_{ss}$.
 The assumptions (GL1),...,(GL3) and (GL7) imply that the essential spectral radius is bounded by $\alpha$ (see e.g. Lemma 2.2 of \cite{BGK}). The assumptions (GL1),...,(GL3) also imply the required continuity properties for the operators acting on $B_s$ and $B_{ss}$, thus we can apply Lemma A.3 of \cite{Ba3} and establish that $\lambda_0$ is a simple eigenvalue of $L_0$ applied on $B_{ss}$ and the associated eigenspace is generated by an eigenvector $v_0\in B_{ss}$.
As  a classical consequence of the assumptions (GL1),...,(GL4) and (GL7)  the spectral stability theorem of \cite{KL}
establishes that for $\delta $ small enough, $\lambda
_{\delta }$ is simple and the associated eigenspaces are generated by $v_\delta \in B_{s}$. Furthermore one has that $\lim_{\delta \rightarrow 0}|\lambda
_{\delta }-\lambda _{0}|\rightarrow 0$ and $||\Pi _{\lambda _{0}}-\Pi
_{\lambda _{\delta } }||_{B_{ss}\rightarrow B_{s}}\rightarrow 0$. Since $\Pi _{\lambda
_{0}}(v_{0})=v_{0}$, $\varphi _{0}(\Pi _{\lambda _{0}}(v_{0}))=1$ and $%
||\Pi _{\lambda _{0}}(v_{0})-\Pi _{\lambda _{\delta }
}(v_{0})||_{s}\rightarrow 0$,
then $v_{\delta }=\frac{\Pi _{\lambda _{\delta
} }(v_{0})}{\varphi _{0}(\Pi _{\lambda _{\delta } }(v_{0}))}$
varies continuously in $B_s$ at $\delta =0$.
This takes care of part (I) and part (II) up to equation (\ref{sstab}).

For the remainder of part (II), putting together $(\ref{1112})$ and \eqref{project3} we have

\begin{eqnarray*}
\Pi _{\lambda _{\delta }} &=&\frac{1}{2\pi i}\int_{\{|z-\lambda
_{\delta }|=\theta \}}(z-L_{\delta })^{-1}dz \\
&=&\frac{1}{2\pi i}\int_{\{|z-\lambda _{\delta }|=\theta
\}}(z-L_{0})^{-1}+\delta (z-L_{0})^{-1}\dot{L}(z-L_{0})^{-1}+O_{\delta } \ dz
\end{eqnarray*}
where $O_{\delta }$ represents an operator such that there is a constant $\tilde{C}$ so that $\frac{||O_{\delta }||_{B_{ss}\rightarrow B_{w}}}{\delta }\le \tilde{C}|\delta|^\eta$
for sufficiently small $\delta$. \ Thus, since for $\delta $ small enough, $\{|z-\lambda _{\delta }|=\theta \}$ also encircles $\lambda_0$
\begin{eqnarray*}
\Pi _{\lambda _{\delta }} &=&\Pi _{\lambda _{0 }}+\delta
\frac{1}{2\pi i}\int_{\{|z-\lambda _{\delta }|=\theta \}}(z-L_{0})^{-1}\dot{L%
}(z-L_{0})^{-1}dz+
\frac{1}{2\pi i}\int_{\{|z-\lambda |=\theta
\}}O_\delta\ dz \\
&=&\Pi _{\lambda _{0 }}+\delta \frac{1}{2\pi i}\int_{\{|z-\lambda
_{\delta }|=\theta \}}(z-L_{0})^{-1}\dot{L}(z-L_{0})^{-1}dz+O_{2,\delta}
\end{eqnarray*}%
where $\frac{||O_{2,\delta}||_{B_{ss}\rightarrow B_{w}}}{\delta }%
\rightarrow 0$ as $\delta \rightarrow 0$.

Since, as proved above, $v_0\in B_{ss}$, we then have
\begin{eqnarray*}
v_{\delta } &=&\frac{\Pi _{\lambda _{\delta } }(v_{0})}{\varphi
_{0}(\Pi _{\lambda _{\delta } }(v_{0}))} \\
&=&\frac{\Pi _{\lambda _{0 }}(v_{0})+\delta \lbrack \frac{1}{2\pi i}%
\int_{\{|z-\lambda _{\delta }|=\theta \}}(z-L_{0})^{-1}\dot{L}%
(z-L_{0})^{-1}v_{0}dz]+q_{\delta }}{\varphi _{0}(\Pi _{\lambda _{\delta
}}(v_{0}))} \\
&=&\frac{v_{0}+\delta \lbrack \frac{1}{2\pi i}\int_{\{|z-\lambda _{\delta
}|=\theta \}}(z-L_{0})^{-1}\dot{L}(z-L_{0})^{-1}v_{0}dz]+q_{\delta }}{\varphi _{0}(\Pi _{\lambda _{\delta
} }(v_{0}))}
\end{eqnarray*}%

where $\frac{||q_{\delta }||_{w}}{\delta }\rightarrow 0$ and using the fact $||\Pi _{\lambda _{0}}-\Pi
_{\lambda _{\delta }}||_{B_{ss}\rightarrow B_{s}}\rightarrow 0$, one has ${\varphi _{0}(\Pi _{\lambda _{\delta
}}(v_{0}))}\to 1$.
Since $%
\lim_{\delta \rightarrow 0}|\lambda _{\delta }-\lambda _{0}|\rightarrow 0$ we have for every fixed $\theta$ small enough \[\lim_{\delta \to 0}\int_{\{|z-\lambda _{\delta }|=\theta \}}(z-L_{0})^{-1}\dot{L}%
(z-L_{0})^{-1}v_{0}dz=\int_{\{|z-\lambda _{0}|=\theta \}}(z-L_{0})^{-1}\dot{L%
}(z-L_{0})^{-1}v_{0}dz.\]
 We then get:%
\begin{equation*}
\frac{v_{\delta }-v_{0}}{\delta }\rightarrow \frac{1}{2\pi i}%
\int_{\{|z-\lambda _{0}|=\theta \}}(z-L_{0})^{-1}\dot{L}(z-L_{0})^{-1}v_{0}dz
\end{equation*}%
in the $B_{w}$ topology.

For the differentiability of $\delta \mapsto \lambda_{\delta}$, let us
consider the normalised eigenvector $v_{\delta }$ as used before. Using that $L_\delta(v_{\delta })=\lambda _{\delta
}v_{\delta }$ we  get%
\begin{equation*}
\frac{L_{\delta }v_{\delta }-L_{0}v_{0}}{\delta }=\frac{\lambda _{\delta
}v_{\delta }-\lambda _{0}v_{0}}{\delta }.
\end{equation*}%
%
Thus
\begin{equation} \label{eq:dereig}
L_{\delta }\frac{v_{\delta }-v_{0}}{\delta }+\frac{L_{\delta }-L_{0}}{\delta
}v_{0}=\lambda_\delta \frac{v_\delta}{\delta} -\lambda_0\frac{v_0}{\delta}=\lambda _{\delta }\frac{v_{\delta }-v_{0}}{\delta }+v_{0}\frac{%
\lambda _{\delta }-\lambda _{0}}{\delta }.
\end{equation}%
To conclude we will argue that both terms on the LHS of \eqref{eq:dereig} and the first term on the RHS of \eqref{eq:dereig} converge in the $B_w$ topology;  this implies that the second term on the RHS of \eqref{eq:dereig} also converges, yielding differentiability of $\delta\mapsto\lambda_\delta$.
Regarding the first term on the LHS of \eqref{eq:dereig}, by (GL4) and \eqref{sstab} we have that $\lim_{\delta \to 0}||L_{\delta }\frac{v_{\delta }-v_{0}}{\delta }-L_{0 }\frac{v_{\delta }-v_{0}}{\delta }||_w=0$;  thus we may replace $L_\delta$ with $L_0$ in this term.
Furthermore, $\frac{L_{\delta }-L_{0}}{\delta
}v_{0}$ converges in $B_w$ by (GL6) and
$\frac{v_{\delta }-v_{0}}{\delta }$ converges in $B_w$ as proved above.
Comparing the LHS and RHS of \eqref{eq:dereig} we see that $\lim_{\delta\to 0}\frac{\lambda _{\delta }-\lambda _{0}}{\delta }$ exists.
\end{proof}

\subsection{An abstract formula for the linear response of the spectrum}
Proposition \ref{lineig1} constructs the abstract formula for the derivative of the eigenvalue.
\begin{proposition}
\label{lineig1}
Suppose the family of transfer operators $L_{\delta }$
satisfy the assumptions (GL1),...,(GL7). 
Suppose $\lambda _{0}$  with $|\lambda _{0}|>\alpha $ is an isolated
eigenvalue for $L_0$ acting on $B_s$ and  that there is $\dot{L}v_{0}\in B_{s}$ such that%
\begin{equation}\label{resex2}
\lim_{\delta \rightarrow 0}\left\|\frac{L_{\delta }-L_{0}}{\delta }v_{0}-\dot{L}%
v_{0}\right\|_{s}=0.
\end{equation}
(as in \ref{resex}).

Then we have the following response formula for the simple isolated eigenvalue
\begin{equation}
\label{dotlameqn}
\dot{\lambda}:=\lim_{\delta \rightarrow 0}\frac{\lambda _{\delta }-\lambda
_{0}}{\delta }=\varphi _{0}(\dot{L}v_{0}).
\end{equation}
\end{proposition}

\begin{proof}
We prove $\lim_{\delta \rightarrow 0}\left\vert \frac{\lambda _{\delta
}-\lambda _{0}}{\delta }-\varphi _{0}(\dot{L}v_{0})\right\vert =0.$
Write
$$
\left\vert \frac{\lambda _{\delta
}-\lambda _{0}}{\delta }-\varphi _{0}(\dot{L}v_{0})\right\vert
\leq \left\vert \frac{\lambda _{\delta
}-\lambda _{0}}{\delta }-\frac{\varphi _{0}(L_{\delta }-L_{0})(v_{0})}{%
\delta }\right\vert +\left\vert \frac{\varphi _{0}(L_{\delta }-L_{0})(v_{0})%
}{\delta }-\varphi _{0}(\dot{L}v_{0})\right\vert.
$$

By the assumption (\ref{resex2}) 
and the fact that $\varphi _{0}\in
B_{s}^{\ast }$  we have $\left\vert \frac{\varphi _{0}(L_{\delta }-L_{0})(v_{0})%
}{\delta }-\varphi _{0}(\dot{L}v_{0})\right\vert \rightarrow 0.$

Let us now consider $\lim_{\delta \rightarrow 0}\left\vert \frac{\lambda
_{\delta }-\lambda _{0}}{\delta }-\frac{\varphi _{0}(L_{\delta
}-L_{0})(v_{0})}{\delta }\right\vert .$ Recalling that by the normalizations introduced before (see the first paragraphs after Corollary \ref{C15}) 
we have $\varphi_0(v_\delta)=1$ for $\delta\in [0,\overline\delta]$, and $\varphi_0(L_0(v_0 ))=\lambda_0\varphi_0(v_0 )=\lambda_0$  and $\varphi_0(L_0(v_\delta ))=\lambda_0$.
We then can write

\begin{eqnarray}
\left\vert \frac{\lambda _{\delta }-\lambda _{0}}{\delta }-\frac{\varphi
_{0}(L_{\delta }-L_{0})(v_{0})}{\delta }\right\vert  &=&\left\vert \frac{%
\lambda _{\delta }-\lambda _{0}}{\delta }-\frac{\varphi _{0}(L_{\delta
}(v_{0}))-\lambda _{0}}{\delta }\right\vert   \notag  \label{diff} \\
&=&\left\vert \frac{\lambda _{\delta }-\varphi _{0}(L_{\delta }(v_{0}))}{%
\delta }\right\vert   \notag \\
&=&\left\vert \frac{\varphi _{0}(L_{\delta }(v_{\delta }))-\varphi
_{0}(L_{\delta }(v_{0}))-[\varphi _{0}(L_{0}(v_{\delta }))-\varphi
_{0}(L_{0}(v_{0}))]}{\delta }\right\vert   \notag \\
&=&\left\vert \frac{\varphi _{0}([L_{\delta }-L_{0}](v_{\delta }-v_{0}))}{%
\delta }\right\vert .
\end{eqnarray}

\bigskip 
We remark that since $\varphi _{0}$ is a left eigenvector 
\begin{equation}
\varphi _{0}([L_{\delta }-L_{0}](v_{\delta }-v_{0}))=\varphi _{0}(\lambda
_{0}^{-n}L_{0}^{n}([L_{\delta }-L_{0}](v_{\delta }-v_{0})))
\label{boundcomp0}\end{equation}

and by (GL2)
\begin{equation}
||L_{0}^{n}([L_{\delta }-L_{0}](v_{\delta }-v_{0}))||_{s}\leq C\alpha
^{n}||[L_{\delta }-L_{0}](v_{\delta }-v_{0})||_{s}+C||[L_{\delta
}-L_{0}](v_{\delta }-v_{0})||_{w}.
\label{boundcomp}
\end{equation}

We are now going to estimate $||[L_{\delta }-L_{0}](v_{\delta }-v_{0})||_{s}$
and $||[L_{\delta }-L_{0}](v_{\delta }-v_{0})||_{w}.$

We remark that $\forall \delta \in \lbrack 0,\overline{\delta })$ and $n\geq
0$ by the fact that $v_{\delta }$ is an eigenvector, and by (GL2) 
\begin{equation*}
||v_{\delta }||_{ss}=||\lambda _{\delta }^{-n}L_{\delta }^{n}(v_{\delta
})||_{ss}\leq C|\lambda _{\delta }^{-n}|\alpha ^{n}||v_{\delta
}||_{ss}+C|\lambda _{\delta }^{-n}|~||v_{\delta }||_{s}.
\end{equation*}%
Since we have $\lambda _{\delta }\rightarrow \lambda _{0}$ and $\alpha
<|\lambda _{0}|$ we can  choose $\hat{\delta}< \overline{\delta }$ \ and  
$n$ such that $C|\lambda _{\delta }^{-n}|\alpha ^{n}<\frac{1}{2}$ and $%
|\lambda _{\delta }^{-n}|<\alpha ^{-n}$ for each $\delta \in \lbrack 0,\hat{%
\delta}).$ 

We then have that  for this choice of $n$, $\forall \delta \in \lbrack 0,\hat{\delta})$%
\begin{equation*}
||v_{\delta }||_{ss}\leq 2C\alpha ^{-n}||v_{\delta }||_{s}
\end{equation*}%
and since by \eqref{sstab} $||v_{\delta }||_{s}$ is uniformly bounded, then also
$||v_{\delta }||_{ss}$ is uniformly bounded for $\delta \in \lbrack
0,\hat{\delta})$. By this $||v_{\delta }-v_{0}||_{ss}$ is also uniformly
bounded, thus by (GL4)%
\begin{equation}
||[L_{\delta }-L_{0}](v_{\delta }-v_{0})||_{s}\leq \delta C\sup_{\delta \in
\lbrack 0,\hat{\delta})}||v_{\delta }-v_{0}||_{ss}. \label{0st}
\end{equation}%

Again, by (GL4)%
\begin{equation*}
||[L_{\delta }-L_{0}](v_{\delta }-v_{0})||_{w}\leq \delta C||v_{\delta
}-v_{0}||_{s}
\end{equation*}%
and since $||v_{\delta }-v_{0}||_{s}\rightarrow 0$ we have that 
\begin{equation}\label{eq46}
    \delta
^{-1}||[L_{\delta }-L_{0}](v_{\delta }-v_{0})||_{w}\rightarrow 0.
\end{equation}

Let us fix some $\epsilon >0,$  by $(\ref{0st}),$ \ there is   $n$ such that $\forall $ $\delta \in
\lbrack 0,\hat{\delta})$
\begin{equation*}
C\lambda _{0}^{-n}\alpha ^{n}||[L_{\delta }-L_{0}](v_{\delta
}-v_{0})||_{s}\leq \delta \frac{\epsilon }{2}.
\end{equation*}

Once $n$ is fixed, by \eqref{eq46}  there is  $\hat{\delta}^{\prime
}\leq \hat{\delta}$ such that $\forall $ $\delta \in \lbrack 0,\hat{\delta}%
^{\prime })$%
\begin{equation}
C\lambda_0^{-n}||[L_{\delta }-L_{0}](v_{\delta }-v_{0})||_{w}\leq \delta \frac{\epsilon }{2}%
\label{1st}.
\end{equation}

Now we are ready to prove that $\frac{|\varphi _{0}([L_{\delta
}-L_{0}](v_{\delta }-v_{0}))|}{\delta }\overset{\delta \rightarrow 0}{%
\rightarrow }0.$ By $(\ref{boundcomp0})$ and \eqref{boundcomp} we can bound $\frac{|\varphi
_{0}([L_{\delta }-L_{0}](v_{\delta }-v_{0}))|}{\delta }$ as follows: choosing $n$ as above, for each $\delta \in
\lbrack 0,\hat{\delta}^{\prime })$ we get 
\begin{eqnarray*}
\delta ^{-1}|\varphi _{0}([L_{\delta }-L_{0}](v_{\delta }-v_{0}))| &=&\delta
^{-1}|\varphi _{0}(\lambda _{0}^{-n}L_{0}^{n}([L_{\delta }-L_{0}](v_{\delta
}-v_{0})))| \\
&\leq &\delta^{-1}||\varphi _{0}||_{B_{s}^{\ast }}||\lambda
_{0}^{-n}L_{0}^{n}([L_{\delta }-L_{0}](v_{\delta }-v_{0}))||_{s} \\
&\leq &||\varphi _{0}||_{B_{s}^{\ast }}\delta ^{-1}[\delta \frac{\epsilon }{2%
}+\delta \frac{\epsilon }{2}] \\
&\leq &\epsilon ||\varphi _{0}||_{B_{s}^{\ast }} 
\end{eqnarray*}
which, since $\epsilon$ is arbitrary,  can be made as small as wanted as $\delta \rightarrow 0,$ proving the
statement.

\end{proof}

\subsection{Verifying abstract transfer operator conditions using properties of expanding maps}

In this section we develop more explicit estimates that allow us to apply the above abstract theory to expanding maps and suitable perturbations  (verifying (A0),(A1),(A2)).
This will lead to the proof of Proposition \ref{Prop:LRmain}.
In the following we will then consider as a stronger and weaker spaces $B_{ss},B_{s},B_{w}$ considered above, the spaces of Borel densities  in a Sobolev space $W^{3,1},W^{1,1},L^1.$ {\color{red}}
The transfer and derivative operators associated to expanding maps will be  defined as acting on  densities as in Definition \ref{1}.
In the following we will recall some basic facts on the properties of such operators. 
\subsubsection{Uniform estimates for individual maps}
\label{sec:unifestmap}
Assumptions (GL1)--(GL3) require the transfer operators to satisfy uniform norm and Lasota--Yorke estimates. These hypotheses will hold when the operators are associated to a uniform family of maps.

\begin{definition}
\label{ufm}A set $U_{M,N}$ of expanding maps of $S^1$ is called a \emph{uniform C}$%
^{k}$\emph{\ family} with parameters $M\geq 0$ and $N>1$  if it satisfies
uniformly the expansiveness and regularity condition: $\forall T\in U_{M,N}$%
\begin{equation*}
||T||_{C^{k}}\leq M,~\inf_{x\in S^{1}}|T^{\prime }(x)|\geq N.
\end{equation*}
\end{definition}

It is well known that the transfer operator associated to a smooth expanding
map has some regularization properties when acting on suitable Sobolev
spaces (see e.g.\ \cite{L2} and \cite{GS}). This is expessed in the following Lemma.

\begin{lemma}[\cite{L2}, Section 1.5;  \cite{GS}, Lemma 29]
\label{Lemsu} Let $U_{M,N}$ be a uniform $C^{k}$ family of expanding maps of $S^1$. 
The transfer operators $L_{T}$ associated to each $T\in U_{M,N}$ satisfy a
uniform Lasota-Yorke inequality on $W^{i,1}(S^{1})$: let  $\alpha=N^{-1} <1$.  For each $
1\leq i\leq k-1$ there are $A_{i},~B_{i}\geq 0$ such that for
each $n\geq 0,$ $T\in U_{M,N}$
\begin{eqnarray}
\|L_{T}^{n}\|_{W^{i-1,1}} &\leq &A_{i}\\
\label{unifLY}||L_{T}^{n}f\Vert _{W^{i,1}} &\leq &\alpha ^{ni}\Vert f\Vert
_{W^{i,1}}+B_{i}\Vert f\Vert _{W^{i-1,1}}\qquad\mbox{ for all $f\in W^{i,1}$.}
\end{eqnarray}
\end{lemma}
Lemma \ref{Lemsu} allow us to establish that the transfer operators associated to expanding maps satisfy the assumptions (GL1),...,(GL3).
These properties, along with the compact immersion of \ $W^{k,1}$ in $%
W^{k-1,1}$ allow one to  classically deduce that the transfer operator $L_{0}$ of a
$C^{k}$ expanding map $T$ is \emph{quasi-compact} on each $W^{i,1}(S%
^{1}),$ with $1\leq i\leq k-1$. Furthermore, by topological transitivity of
expanding maps, $1$ is the only eigenvalue on the unit circle.
All of the above leads to the following classical result; see e.g. \cite{L2}, Section 3 or \cite{GS} Proposition 30.
Define the spaces $V_{i}:=\{g\in
W^{i,1}(S^{1})~s.t.~\int_{S^{1}}g~dm=0\}$ for $i=0,1$.
\begin{proposition}
\label{propora} \ There is $C\geq 0$ and $0<\rho<1$ 
such that for each $1\leq
i\leq k-1$, $g\in V_{i}$, and $n\geq 0$ one has%
\begin{equation*}
\Vert L_0^{n}g\Vert _{W^{i,1}}\leq C\rho ^{n}\Vert g\Vert _{W^{i,1}}.
\end{equation*}%
In particular, the resolvent $(Id-L_0)^{-1}:=\sum_{j=0}^{\infty }L_0^{j}$
is a well-defined and bounded operator on $V_{i}$.
\end{proposition}


\subsubsection{Small perturbation estimates}

In this subsection we recall some more or less known estimates (see e.g.\ \cite{Gdisp}) showing that a small
perturbation of an expanding map induces a small perturbation of the
associated transfer operator when considered as acting from a stronger to a
weaker Sobolev space.
This will allow us to verify that the assumption (GL4) applies to suitable deterministic perturbations of expanding maps.

\begin{proposition}
\label{prop14} Let
$\{T_{\delta }\}_{\delta \in [0,\overline{\delta })}$ 
be
a  family of $C^{2}$ expanding maps such that $T_{0}\in C^{3}$. Let  $%
L_{\delta }$ be the transfer operators associated to $T_{\delta }$ and suppose
that for some $K\in \mathbb{R}$
one has
\begin{equation*}
||T_{\delta }-T_{0}||_{C^{2}}\leq K\delta. 
\end{equation*}%
Then there is a $C>0$\ such that $\forall f\in W^{1,1}$:%
\begin{eqnarray}
||(L_{\delta }-L_{0})f||_{L^{1}} &\leq &\delta C||f||_{W^{1,1}}  \label{2l}
\\
||(L_{\delta }-L_{0})f||_{W^{1,1}} &\leq &\delta C||f||_{W^{2,1}} . 
\label{22}
\label{3l}
\end{eqnarray}
\end{proposition}

\begin{proof}
In \cite{Gdisp}, Section 7 it is proved that if $L_{0}$ and $L_{\delta }$
are transfer operators of $C^{3}$ expanding maps $T_{0}$ and $T_{\delta },$
such that for some $K\in \mathbb{R}$ 
\begin{equation*}
||T_{\delta }-T_{0}||_{C^{2}}\leq K\delta 
\end{equation*}
then there is a $C\in \mathbb{R}$\ such that $\forall f\in W^{1,1}$:%
\begin{equation}\label{54}
||(L_{\delta }-L_{0})f||_{1}\leq \delta C||f||_{W^{1,1}}
\end{equation}%
and \eqref{2l}  is established. From this we can also recover \eqref{22},
using the explicit formula for the transfer operator:%
\begin{equation}
\lbrack L_{0}f](x)=\sum_{y\in T_0^{-1}(x)}\frac{f(y)}{|T_{0}^{\prime }(y)|}.
\label{pf}
\end{equation}%
Noting that $T_{0}^{\prime }(y)=T_{0}^{\prime }(T_{0}^{-1}(x))$ we
can compute the derivative of $(\ref{pf})$ 
\begin{equation*}
(L_{0}f)^{^{\prime }}=\sum_{y\in T_{0}^{-1}(x)}\frac{1}{(T_{0}^{\prime
}(y))^{2}}f^{\prime }(y)-\frac{T_{0}^{\prime \prime }(y)}{(T_{0}^{\prime
}(y))^{3}}f(y).
\end{equation*}%
And similarly for $L_{\delta }.$ Note that%
\begin{equation}
(L_{0}f)^{^{\prime }}=L_{0}(\frac{1}{T_{0}^{\prime }}f^{\prime })-L_{0}(
\frac{T_{0}^{\prime \prime }}{(T_{0}^{^{\prime }})^{2}}f).  \label{preLY}
\end{equation}

Hence applying \eqref{54},  \eqref{preLY}  and the fact that $L$ is a weak contraction on $L^1$
\begin{eqnarray*}
||(L_{\delta }-L_{0})f||_{W^{1,1}} &\leq &||(L_{\delta
}-L_{0})f||_{L^{1}}+||((L_{\delta }-L_{0})f)^{\prime }||_{L^{1}} \\
&\leq &\delta C||f||_{W^{1,1}}+||(L_{\delta }-L_{0})(\frac{1}{T_{0}^{\prime }%
}f^{\prime })-(L_{\delta }-L_{0})(\frac{T_{0}^{\prime \prime }}{%
(T_{0}^{^{\prime }})^{2}}f)||_{L^{1}} \\
&&+||(\frac{1}{T_{\delta }^{\prime }}f^{\prime })-(\frac{1}{T_{0}^{\prime }}%
f^{\prime })+(\frac{T_{\delta }^{\prime \prime }}{(T_{\delta }^{^{\prime
}})^{2}}f)-(\frac{T_{0}^{\prime \prime }}{(T_{0}^{^{\prime }})^{2}}%
f)||_{L^{1}} \\
&\leq &\delta C(||f||_{W^{1,1}}+||\frac{1}{T_{0}^{\prime }}f^{\prime
}||_{W^{1,1}}+||\frac{T_{0}^{\prime \prime }}{(T_{0}^{^{\prime }})^{2}}%
f||_{W^{1,1}}) \\
&&+C_1||T_{\delta }-T_{0}||_{C^{2}}||f||_{W^{1,1}} \\
&\leq &\delta C_{2}||f||_{W^{2,1}}
\end{eqnarray*}%
for some $C_1, C_{2}\geq 0$ depending on $T_{0}$ but not on $f.$ This proves $(%
\ref{22}).$

\end{proof}

Small perturbations in a mixed norm sense -- namely $W^{1,1}$ into $L^{1}$ -- (see (GL4) and Proposition \ref{prop14}) 
will imply a classical fact: the stability of the
resolvent.
The following will be used in the proof of Proposition \ref{Prop:LRmain} to verify  assumption \eqref{reslim} when invoking Theorem \ref{LR}.

\begin{proposition}
\label{resstab}
For $\delta \in
[0,\overline{\delta})$, let $T_{\delta }:S^{1}\rightarrow S^{1}$  be a family of $C^{3}$ expanding maps. Suppose that the
dependence of the family on $\delta $ is differentiable at $0$ in the sense of assumptions (A0), (A1),(A2), then
\begin{equation}
\lim_{\delta \rightarrow 0}||(Id-L_{\delta
})^{-1}-(Id-L_{0})^{-1}||_{V_{1}\rightarrow V_{0}}=0.  \label{n1}
\end{equation}
\end{proposition}

\begin{proof}
The result  follows from Theorem 1 of \cite{KL}. This theorem says that if a family of bounded linear operators $L_\delta$ acting on  weak and strong Banach spaces $\mathcal{B}_s,\mathcal{B}_w$ such that (i) there is a compact immersion of $\mathcal{B}_s$ into $\mathcal{B}_w$,  (ii) the family of operators satisfy (GL1), the first equation of (GL2), (GL3) and (GL4), and (iii) the  spectral radius of $L_0:\mathcal{B}_s\to \mathcal{B}_s$ is strictly smaller than $1$,
\begin{equation*}
\lim_{\delta \rightarrow 0}||(Id-L_{\delta
})^{-1}-(Id-L_{0})^{-1}||_{\mathcal{B}_s\rightarrow \mathcal{B}_w}=0.  
\end{equation*}
Equation $(\ref{n1})$ will follow directly by applying this theorem to the transfer operators $L_\delta$ associated to a family of expanding maps $T_\delta$ satisfying (A0), (A1), (A2) considering $\mathcal{B}_s=V_{1}$ and $\mathcal{B}_w=V_0$.
For (i) the compact immersion is well known from the Rellich--Kondracov theorem. 
For (ii), note that for the family of maps $T_\delta$,
(GL1), (GL2) and (GL3) are verified
by the Lasota--Yorke inequalities established in Lemma \ref{Lemsu}, 
and
the small perturbation assumptions needed at (GL4)  are established in Proposition\ \ref{prop14}.
For (iii), we note that by Proposition \ref{propora}, the spectral radius of $L_0|_{V_1}\le\rho<1$.
Thus, noting that $\|(Id-L_0)^{-1}\|_{V_1}<\infty$ (also by Proposition \ref{propora}) the theorem can be applied, implying $(\ref{n1})$.

\end{proof}

\subsubsection{The derivative operator}\label{9.4.2}

In this section we study the operator $\dot{L}$ representing a \textquotedblleft first
derivative\textquotedblright\ of the transfer operator with respect to the
perturbation parameter $\delta$.  
The next result is similar to \cite[Proposition 3.1]{GP} but has been strengthened to be
quantitative and uniform, in order to 
verify the assumptions  (GL5), (GL6) of Theorem \ref{GL}  and the assumption \eqref{resex} of Proposition \ref{lineig1}.


\begin{proposition}
\label{mainprop} Let $T_{\delta }:S^{1}\rightarrow S^{1}$, where $\delta \in
[0,\overline{\delta})$  be a family of $C^{3}$ expanding maps. Suppose that the dependence of the family on $\delta $ is  differentiable at $0$
in the sense of (A2).
Let $L_\delta$ be the transfer operator associated to $T_\delta$.
Let us consider again the operator  $\dot{L}:W^{1,1}\to L^1$ defined in 
\eqref{der} by
\begin{equation}\label{formulaz}
\dot{L}(w)=-L_{0}\left( \frac{w\dot{T} ^{\prime }}{T_{0}^{\prime }}%
\right) -L_{0}\left( \frac{\dot{T} w^{\prime }}{T_{0}^{\prime }}\right)
+L_{0}\left( \frac{\dot{T} T_{0}^{\prime \prime }}{T_{0}^{^{\prime }2}}%
w\right).
\end{equation}

For $w\in C^2(S^1,\mathbb{R})$  we have 
\begin{equation}\label{58}
\dot{L}(w)=\lim_{\delta \rightarrow 0}\left( \frac{L_{\delta }w-L_{0}w%
}{\delta }\right)
\end{equation}
with convergence in the $C^1$ topology (and therefore also in the $W^{1,1}$ and $H^{1}$ topologies). 
Moreover, one has a quantitative and uniform convergence on the unit ball $B_{W^{2,1}(S^1)}$:
\begin{equation}
\sup_{w\in B_{W^{2,1}(S^1)}}\left\|\dot{L}(w)-\frac{(L_{\delta}-L_{0})w}{\delta }%
\right\|_{L^1}=O(\delta).  \label{GL6}
\end{equation}
\end{proposition}

We denote by $\{y_{i}^{\delta }\}_{i=1}^{d}:=T_{\delta }^{-1}(x)$ and $%
\{y_{i}^{0}\}_{i=1}^{d}:=T_{0}^{-1}(x)$ the $d$ preimages under $T_{\delta }$
and $T_{0}$, respectively, of a point $x\in X$. Futhermore, we assume that
the indexing is chosen so that $y_{i}^{\delta }$ is a small perturbation of $%
y_{i}^{0}$, for $1\leq i\leq d$. Before presenting the proof of Proposition %
\ref{mainprop} we state a lemma.


\begin{lemma}
\label{tec} For $y_{i}^{\delta }\in T_{\delta }^{-1}(x)$ we can write 
\begin{equation*}
y_{i}^{\delta }=y_{i}^{0}+\delta \left( -\frac{\dot{T}(y_{i}^{0})}{%
T_{0}^{\prime }(y_{i}^{0})}\right) +O_{C^{1}}(\delta ^{2}),\quad 
\mbox{in the limit
as $\delta\to 0$},
\end{equation*}
where we say $F_\delta$ is $O_{C^1}(\delta^2)$ if
\begin{equation}
\underset{\delta \in (0,\overline{\delta })}{\sup }\frac{\|F_\delta\|_{C^1}}{\delta ^{2}}<+\infty.  \label{OO}
\end{equation}
\end{lemma}

\begin{proof}[Proof of Lemma \protect\ref{tec}]
For a family of functions $F_{\delta }$ we will say that $F_{\delta
}=O_{C^{0}}(\delta ^{2})$ if $\underset{\delta \in (0,\overline{\delta })}{%
\sup }\frac{||F_{\delta }||_{C^{0}}}{\delta ^{2}}<+\infty $ for some $%
\overline{\delta }>0.$ Let $T_{\delta ,i}$ be the branches of $T_{\delta }$.
We have $y_{i}^{\delta }(x)=T_{\delta ,i}^{-1}(x)$ and%
\begin{equation}
\frac{dy_{i}^{\delta }(x)}{dx}=\frac{1}{T_{\delta }^{\prime }(T_{\delta
,i}^{-1}(x))}.  \label{part1}
\end{equation}

Now let us compute $\frac{dy_{i}^{\delta }(x)}{d\delta }.$ \ Let us fix $%
x\in S^{1}$ and write 
\begin{equation}
y_{i}^{\delta }(x)=y_{i}^{0}(x)+\delta \epsilon _{i}(x)+F_{i}(\delta ,x)
\label{sss}
\end{equation}%
where for each $x$, $F_{i}(\delta ,x)=o(\delta )$ as $\delta \rightarrow 0$.
We will show that $\epsilon _{i}(x)=-\frac{\dot{T}(y_{i}^{0}(x))}{%
T_{0}^{\prime }(y_{i}^{0}(x))}$ and $F_{i}(\delta ,x)=O(\delta ^{2})$ for
each $x,$ then we will show that $F_{i}(\delta ,x)$ satisfies \eqref{OO}.

For the first two claims, let us fix $x\in S^{1}.$ Substituting $(\ref{sss})$
into the identity $T_{\delta }(y_{i}^{\delta }(x))=x$ we can expand%
\begin{eqnarray*}
x &=&T_{\delta }(y_{i}^{\delta }(x)) \\
&=&T_{0}(y_{i}^{\delta }(x))+\delta \dot{T}(y_{i}^{\delta }(x))+E(\delta
,y_{i}^{\delta }(x))
\end{eqnarray*}%
where by $(\ref{not492})$, $E(\delta ,x)$ satisfies \eqref{OO}. Further, 
\begin{eqnarray*}
x &=&T_{0}(y_{i}^{0}(x)+\delta \epsilon _{i}(x)+F_{i}(\delta ,x))
\label{long} \\
&&\ \ \ +\delta \dot{T}(y_{i}^{0}(x)+\delta \epsilon _{i}(x)+F_{i}(\delta
,x))+E(\delta ,y_{i}^{\delta }(x)).
\end{eqnarray*}%
Since $T_{0}\in C^{2}$ we can write the first term in the right hand side of 
$(\ref{long})$ as%
\begin{eqnarray}
T_{0}(y_{i}^{0}(x)+\delta \epsilon _{i}(x)+F_{i}(\delta ,x))
&=&  T_{0}(y_{i}^{0}(x))+T_{0}^{\prime }(y_{i}^{0}(x))(\delta \epsilon
_{i}(x)+F_{i}(\delta ,x)) \\
&+\frac{1}{2}T_{0}^{^{\prime \prime }}(\xi )&(\delta \epsilon
_{i}(x)+F_{i}(\delta ,x))^{2}  \label{aaa}
\nonumber
\end{eqnarray}

for some $\xi \in S^{1}$. Since $\dot{T}\in C^{2}$ we can develop the second
term in \eqref{long} as 
\begin{eqnarray}
\delta \dot{T}(y_{i}^{0}(x)+\delta \epsilon _{i}(x)+F_{i}(\delta ,x))
&=&\delta \dot{T}(y_{i}^{0}(x))+\delta \dot{T}^{\prime
}(y_{i}^{0}(x))(\delta \epsilon _{i}(x)+F_{i}(\delta ,x)) \\
&&+\frac{1}{2}\dot{T}_{0}^{^{\prime \prime }}(\xi ^{\prime })(\delta
\epsilon _{i}(x)+F_{i}(\delta ,x))^{2}  \label{ccc}
\nonumber
\end{eqnarray}%
for some $\xi ^{\prime }\in S^{1}$ and use that $T_{0}(y_{i}^{0}(x))=x$ to
cancel terms on either side of $(\ref{long})$ and get%
\begin{eqnarray}
0 &=&T_{0}^{\prime }(y_{i}^{0}(x))(\delta \epsilon _{i}(x)+F_{i}(\delta
,x))+\delta \dot{T}(y_{i}^{0}(x))+\delta \dot{T}^{\prime
}(y_{i}^{0}(x))(\delta \epsilon _{i}(x)+F_{i}(\delta ,x))  \label{hent} \\
&&+\frac{1}{2}T_{0}^{^{\prime \prime }}(\xi ^{\prime })(\delta \epsilon
_{i}(x)+F_{i}(\delta ,x))^{2}+\frac{1}{2}\dot{T}_{0}^{^{\prime \prime }}(\xi
^{\prime })(\delta \epsilon _{i}(x)+F_{i}(\delta ,x))^{2}+E(\delta
,y_{i}^{\delta }(x)).  \notag
\end{eqnarray}

Recalling that $T_{0}^{^{\prime \prime }}$ is uniformly bounded on $S^{1},$
for each fixed $x,$ as $\delta \rightarrow 0$ we can then identify the
first-order terms in $(\ref{hent})$ as $\delta T_{0}^{\prime
}(y_{i}^{0}(x))\epsilon _{i}(x)+\delta \dot{T}(y_{i}^{0}(x))$, \ giving $%
\epsilon _{i}(x)=-\frac{\dot{T}(y_{i}^{0})}{T_{0}^{\prime }(y_{i}^{0})}$ \
and thus 
\begin{equation*}
\frac{dy_{i}^{\delta }(x)}{d\delta }=-\frac{\dot{T}(y_{i}^{0}(x))}{%
T_{0}^{\prime }(y_{i}^{0}(x))}.
\end{equation*}

Putting $\epsilon _{i}(x)=-\frac{\dot{T}\circ y_{i}^{0}}{T_{0}^{\prime
}\circ y_{i}^{0}}$ into $(\ref{hent})$ we get%
\begin{eqnarray*}
0 &=&T_{0}^{\prime }(y_{i}^{0}(x))F_{i}(\delta ,x)+\delta \dot{T}^{\prime
}(y_{i}^{0}(x))F_{i}(\delta ,x) \\
&&+\frac{1}{2}T_{0}^{^{\prime \prime }}(\xi ^{\prime })(\delta \epsilon
_{i}(x)+F_{i}(\delta ,x))^{2}+\frac{1}{2}\dot{T}_{0}^{^{\prime \prime }}(\xi
^{\prime })(\delta \epsilon _{i}(x)+F_{i}(\delta ,x))^{2}+E(\delta
,y_{i}^{\delta }(x))
\end{eqnarray*}

from which%
\begin{equation*}
|T_{0}^{\prime }(y_{i}^{0}(x))F_{i}(\delta ,x)+\delta \dot{T}^{\prime
}(y_{i}^{0}(x))F_{i}(\delta ,x)|\leq |E(\delta ,y_{i}^{\delta
}(x))|+M(\delta \epsilon _{i}(x)+F_{i}(\delta ,x))^{2}
\end{equation*}

where $M=\sup |\max (T_{0}^{^{\prime \prime }},\dot{T}_{0}^{^{\prime \prime
}})|.$ Considering $\delta $ small enough we can suppose $T_{0}^{\prime
}(y_{i}^{0}(x))+\delta \dot{T}^{\prime }(y_{i}^{0}(x))\geq 1,$ hence%
\begin{equation}\label{neq1}
|F_{i}(\delta ,x)|\leq |E(\delta ,y_{i}^{\delta }(x))|+M(\delta ^{2}\epsilon
_{i}(x)+F_{i}(\delta ,x)^{2}+2\delta \epsilon _{i}(x)F_{i}(\delta ,x)).
\end{equation}

Since $F_{i}(\delta ,x)=o(\delta )$ we get that for each $x$%
\begin{equation}
F_{i}(\delta ,x)=O(\delta ^{2}).  \label{o2}
\end{equation}

By $(\ref{sss})$ we have 
\begin{equation*}
F_{i}(\delta ,x)=y_{i}^{\delta }(x)-y_{i}^{0}(x)+\delta \frac{\dot{T}%
(y_{i}^{0}(x))}{T_{0}^{\prime }(y_{i}^{0}(x))}
\end{equation*}%
and by $($\ref{part1}$)$%
\begin{equation}
\frac{dF_{i}(\delta ,x)}{dx}=\frac{1}{T_{\delta }^{\prime }(y_{i}^{\delta
}(x))}-\frac{1}{T_{0}^{\prime }(y_{i}^{0}(x))}+\delta \frac{%
y_{i}^{0}(x)^{\prime }\dot{T}^{\prime }(y_{i}^{0}(x))T_{0}^{\prime
}(y_{i}^{0}(x))-y_{i}^{0}(x)^{\prime }\dot{T}(y_{i}^{0}(x))T_{0}^{^{\prime
\prime }}(y_{i}^{0}(x))}{(T_{0}^{\prime }(y_{i}^{0}(x)))^{2}}.  \label{part3}
\end{equation}%
This shows that $\frac{dF_{i}(\delta ,x)}{dx}$ is uniformly bounded and then
by $(\ref{o2})$ and the compactness of $S^{1}$%
\begin{equation*}
\lim_{\delta \rightarrow 0}||F_{i}(\delta ,x)||_{\infty }=0.
\end{equation*}%
for each $i.$ \ This means that for $\delta $ small enough $|F_{i}(\delta
,x)|\leq \frac{1}{2M}$ for each $x$ and then inserting this in $(\ref{neq1})$
we get%
\begin{equation*}
|F_{i}(\delta ,x)|\leq |E(\delta ,y_{i}^{\delta }(x))|+M(\delta ^{2}\epsilon
_{i}(x)+\frac{1}{2M}|F_{i}(\delta ,x)|+2\delta \epsilon _{i}(x)|F_{i}(\delta
,x)|)
\end{equation*}%
and 
\begin{equation*}
|F_{i}(\delta ,x)|\leq 2|E(\delta ,y_{i}^{\delta }(x))|+2M(\delta
^{2}\epsilon _{i}(x)+2\delta \epsilon _{i}(x)|F_{i}(\delta ,x)|)
\end{equation*}%
\begin{equation*}
(1-4M\delta \epsilon _{i}(x))|F_{i}(\delta ,x)|\leq 2|E(\delta
,y_{i}^{\delta }(x))|+2M(\delta ^{2}\epsilon _{i}(x))
\end{equation*}%
\begin{equation*}
|F_{i}(\delta ,x)|\leq \frac{2|E(\delta ,y_{i}^{\delta }(x))|+2M(\delta
^{2}\epsilon _{i}(x))}{(1-4M\delta \epsilon _{i}(x))}
\end{equation*}

from which we get that%
\begin{equation}
F_{i}(\delta ,x)=O_{C^{0}}(\delta ^{2}).
\end{equation}

Now we prove that 
\begin{equation*}
\frac{dF_{i}(\delta ,x)}{dx}=O_{C^{0}}(\delta ^{2}).
\end{equation*}

Recalling $(\ref{part3})$ we have%
\begin{equation}
\frac{dF_{i}(\delta ,x)}{dx}=\frac{1}{T_{\delta }^{\prime }(y_{i}^{\delta
}(x))}-\frac{1}{T_{0}^{\prime }(y_{i}^{0}(x))}+\delta \frac{%
y_{i}^{0}(x)^{\prime }\dot{T}^{\prime }(y_{i}^{0}(x))T_{0}^{\prime
}(y_{i}^{0}(x))-y_{i}^{0}(x)^{\prime }\dot{T}(y_{i}^{0}(x))T_{0}^{^{\prime
\prime }}(y_{i}^{0}(x))}{(T_{0}^{\prime }(y_{i}^{0}(x)))^{2}}.
\end{equation}

Now we analyze the fist part of the right hand side%
\begin{equation}\label{for_later}
\frac{1}{T_{\delta }^{\prime }(y_{i}^{\delta }(x))}-\frac{1}{T_{0}^{\prime
}(y_{i}^{0}(x))}=[\frac{1}{T_{\delta }^{\prime }(y_{i}^{\delta }(x))}-\frac{1%
}{T_{\delta }^{\prime }(y_{i}^{0}(x))}]+[\frac{1}{T_{\delta }^{\prime
}(y_{i}^{0}(x))}-\frac{1}{T_{0}^{\prime }(y_{i}^{0}(x))}]
\end{equation}

where for the first summand, since $y_{i}^{\delta }(x)=y_{i}^{0}(x)+\delta
\epsilon _{i}(x)+F_{i}(\delta ,x)$ we get%
\begin{eqnarray*}
\lbrack \frac{1}{T_{\delta }^{\prime }(y_{i}^{\delta }(x))}-\frac{1}{%
T_{\delta }^{\prime }(y_{i}^{0}(x))}] &=&[\frac{1}{T_{\delta }^{\prime
}(y_{i}^{0}(x)+\delta \epsilon _{i}(x)+F_{i}(\delta ,x))}-\frac{1}{T_{\delta
}^{\prime }(y_{i}^{0}(x))}] \\
&=&\frac{-T_{\delta }^{^{\prime \prime }}(y_{i}^{0}(x))}{T_{\delta }^{\prime
}(y_{i}^{0}(x))^{2}}[\delta \epsilon _{i}(x)+F_{i}(\delta ,x)] \\
&&+\frac{d}{dx}[\frac{-T_{\delta }^{^{\prime \prime }}}{T_{\delta }^{\prime
}{}^{2}}](\xi )\frac{1}{2}[\delta \epsilon _{i}(x)+F_{i}(\delta ,x)]^{2}
\end{eqnarray*}

for some $\xi \in S^{1}$ (depending on $x$), by Lagrange theorem. Then we get%
\begin{eqnarray*}
\lbrack \frac{1}{T_{\delta }^{\prime }(y_{i}^{\delta }(x))}-\frac{1}{%
T_{\delta }^{\prime }(y_{i}^{0}(x))}] &=&\frac{-T_{\delta }^{^{\prime \prime
}}(y_{i}^{0}(x))}{T_{\delta }^{\prime }(y_{i}^{0}(x))^{2}}[\delta \frac{-%
\dot{T}(y_{i}^{0}(x))}{T_{0}^{\prime }(y_{i}^{0}(x))}+F_{i}(\delta ,x)] \\
&&+\frac{d}{dx}[\frac{-T_{\delta }^{^{\prime \prime }}}{T_{\delta }^{\prime
}{}^{2}}](\xi )[\delta \epsilon _{i}(x)+F_{i}(\delta ,x)]^{2} \\
&=&\frac{-[T_{0}^{\prime \prime }(y_{i}^{0}(x))+\delta \dot{T}^{\prime
\prime }(y_{i}^{0}(x))+E^{\prime \prime }(\delta ,y_{i}^{0}(x))]}{%
[T_{0}^{\prime }(y_{i}^{0}(x))+\delta \dot{T}^{\prime
}(y_{i}^{0}(x))+E^{\prime }(\delta ,y_{i}^{0}(x))]^{2}}[\delta \frac{-\dot{T}%
(y_{i}^{0}(x))}{T_{0}^{\prime }(y_{i}^{0}(x))}+F_{i}(\delta ,x)] \\
&&+\frac{d}{dx}[\frac{-T_{\delta }^{^{\prime \prime }}}{T_{\delta }^{\prime
}{}^{2}}](\xi )\frac{1}{2}([\delta \epsilon _{i}(x)+F_{i}(\delta ,x)]^{2}).
\end{eqnarray*}%
Since \ $\frac{d}{dx}[\frac{-T_{\delta }^{^{\prime \prime }}}{T_{\delta
}^{\prime }{}^{2}}]$ is bounded we get%
\begin{eqnarray*}
\lbrack \frac{1}{T_{\delta }^{\prime }(y_{i}^{\delta }(x))}-\frac{1}{%
T_{\delta }^{\prime }(y_{i}^{0}(x))}] &=&\frac{\delta \dot{T}%
(y_{i}^{0}(x))[T_{0}^{\prime \prime }(y_{i}^{0}(x))+\delta \dot{T}^{\prime
\prime }(y_{i}^{0}(x))+E^{\prime \prime }(\delta ,y_{i}^{0}(x))]}{%
T_{0}^{\prime }(y_{i}^{0}(x))[T_{0}^{\prime }(y_{i}^{0}(x))+\delta \dot{T}%
^{\prime }(y_{i}^{0}(x))+E^{\prime }(\delta ,y_{i}^{0}(x))]^{2}}%
+O_{C^{0}}(\delta ^{2}) \\
&=&\frac{\delta T_{0}^{\prime \prime }(y_{i}^{0}(x))\dot{T}(y_{i}^{0}(x))}{%
[T_{0}^{\prime }(y_{i}^{0}(x))]^{3}+2\delta \dot{T}^{\prime
}(y_{i}^{0}(x))T_{0}^{\prime }(y_{i}^{0}(x))+O_{C^{0}}(\delta ^{2})}%
+O_{C^{0}}(\delta ^{2}) \\
&=&\frac{\delta T_{0}^{\prime \prime }(y_{i}^{0}(x))\dot{T}(y_{i}^{0}(x))}{%
[T_{0}^{\prime }(y_{i}^{0}(x))]^{3}}+O_{C^{0}}(\delta ^{2}).
\end{eqnarray*}

The second summand is expanded similarly as follows%
\begin{eqnarray}
\lbrack \frac{1}{T_{\delta }^{\prime }(y_{i}^{0}(x))}-\frac{1}{T_{0}^{\prime
}(y_{i}^{0}(x))}] &=&\nonumber [\frac{1}{T_{0}^{\prime }(y_{i}^{0}(x))+\delta \dot{T}%
^{\prime }(y_{i}^{0}(x))+E^{\prime }(\delta ,y_{i}^{0}(x))}-\frac{1}{%
T_{0}^{\prime }(y_{i}^{0}(x))}] \\
&=& \nonumber
\frac{-1}{T_{0}^{\prime }(y_{i}^{0}(x))^{2}}[\delta \dot{T}^{\prime
}(y_{i}^{0}(x))+E^{\prime }(\delta ,y_{i}^{0}(x))] \\
&& \nonumber +\frac{1}{2}\frac{2}{(\xi )^{3}}[\delta \dot{T}^{\prime
}(y_{i}^{0}(x))+E^{\prime }(\delta ,y_{i}^{0}(x))]^{2} \\
&=& \label{for_later2}\frac{-\delta \dot{T}^{\prime }(y_{i}^{0}(x))T_{0}^{\prime }(y_{i}^{0}(x))%
}{T_{0}^{\prime }(y_{i}^{0}(x))^{3}}+O_{C^{0}}(\delta ^{2}).
\end{eqnarray}

since applying the Lagrange theorem $|\xi -T_{0}^{\prime }(y_{i}^{0}(x))|$
can be small as wanted when $\delta $ is small enough and then $\xi $ is
bounded away from $0$. 

Putting the two summands expanded together and recalling that $\frac{%
dy_{i}^{\delta }(x)}{dx}=\frac{1}{T_{\delta }^{\prime }(T_{\delta
,i}^{-1}(x))}=\frac{1}{T_{\delta }^{\prime }(y_{i}^{\delta}(x))}$ we get 
\begin{eqnarray}\label{xxx}
\frac{dF_{i}(\delta ,x)}{dx} &=&\frac{\delta T_{0}^{\prime \prime
}(y_{i}^{0}(x))\dot{T}(y_{i}^{0}(x))}{[T_{0}^{\prime }(y_{i}^{0}(x))]^{3}}-%
\frac{\delta \dot{T}^{\prime }(y_{i}^{0}(x))T_{0}^{\prime }(y_{i}^{0}(x))}{%
T_{0}^{\prime }(y_{i}^{0}(x))^{3}} \\
&&\nonumber+\delta \frac{d}{dx}[\frac{\dot{T}(y_{i}^{0}(x))}{T_{0}^{\prime
}(y_{i}^{0}(x))}]+O_{C^{0}}(\delta ^{2}) \\
&=&\nonumber \frac{\delta T_{0}^{\prime \prime }(y_{i}^{0}(x))\dot{T}(y_{i}^{0}(x))}{%
[T_{0}^{\prime }(y_{i}^{0}(x))]^{3}}-\frac{\delta \dot{T}^{\prime
}(y_{i}^{0}(x))T_{0}^{\prime }(y_{i}^{0}(x))}{T_{0}^{\prime
}(y_{i}^{0}(x))^{3}} \\
&&\nonumber +\delta \frac{y_{i}^{0}(x)^{\prime }\dot{T}^{\prime
}(y_{i}^{0}(x))T_{0}^{\prime }(y_{i}^{0}(x))-y_{i}^{0}(x)^{\prime
}T_{0}^{^{\prime \prime }}(y_{i}^{0}(x))\dot{T}(y_{i}^{0}(x))}{%
(T_{0}^{\prime }(y_{i}^{0}(x)))^{2}}+O_{C^{0}}(\delta ^{2}) \\
&=& \label{xxx2}
O_{C^{0}}(\delta ^{2})
\end{eqnarray}%
proving the statement.
\end{proof}

The following simple lemma on the convergence of compositions of functions
will also be useful.

\begin{lemma}
\label{Lclaim} 
Let $\overline{\delta}>0, K_0>0, K_1>0.$ Suppose the following hold when $\delta\in[0,\overline{\delta})$:

$f_{\delta},g_{\delta}\in C^1(S^{1},\mathbb{R})$, $||f_{\delta }||_{C^{1}}\le K_0$, $||g_{\delta }||_{C^{1}}\le K_0$,  $||f_{\delta }-f_{0}||_{C^{1}}\rightarrow 0 $ and $ ||g_{\delta
}-g_{0}||_{C^{1}}\rightarrow 0$.
Then%
\begin{equation}\label{claim00}
\lim_{\delta\to 0}|| f_{\delta }\circ g_{\delta }-f_{0}\circ g_{0}||_{C^{1}}=0.
\end{equation}
If, in addition,
$||f_{\delta }-f_{0}||_{C^{0}}\leq K_{1}\delta ,$ and $||g_{\delta}-g_{0}||_{C^{0}}\leq K_{1}\delta $  then there is $%
K_{2}\geq 0$ such that for $\delta \in \lbrack 0,\eta ]$%
\begin{equation}
|| f_{\delta }\circ g_{\delta }-f_{0}\circ g_{0}||_{C^{0}}\leq K_{2}\delta
\label{claim0}
\end{equation}%
with $K_{2}$ only depending on $Lip(f_{0})$ and $K_{1}.$
\end{lemma}

\begin{proof}
To prove $(\ref{claim00})$ we can write
\begin{eqnarray*}
||\ (f_{\delta }\circ g_{\delta }-f_{0}\circ g_{0})^{\prime }||_{\infty }
&=&||\ (f_{\delta }\circ g_{\delta }-f_{0}\circ g_{\delta }+f_{0}\circ
g_{\delta }-f_{0}\circ g_{0})^{\prime }||_{\infty } \\
&\le&||\ (f_{\delta }\circ g_{\delta }-f_{0}\circ g_{\delta })^{\prime
}||_{\infty }+||(f_{0}\circ g_{\delta }-f_{0}\circ g_{0})^{\prime
}||_{\infty }
\end{eqnarray*}%
and while it is obvious that the first summand goes to $0$, for the second we have
\begin{eqnarray*}
&&||(f_{0}\circ g_{\delta }-f_{0}\circ g_{0})^{\prime }||_{\infty
}=||g_{\delta }^{\prime }f_{0}^{\prime }\circ g_{\delta }-g_{0}^{\prime
}f_{0}^{\prime }\circ g_{0}||_{\infty } \\
&\leq &||g_{\delta }^{\prime }f_{0}^{\prime }\circ g_{\delta }-g_{\delta
}^{\prime }f_{0}^{\prime }\circ g_{0}||_{\infty }+||g_{\delta }^{\prime
}f_{0}^{\prime }\circ g_{0}-g_{0}^{\prime }f_{0}^{\prime }\circ
g_{0}||_{\infty }  \notag
\end{eqnarray*}

where $||g_{\delta }^{\prime }f_{0}^{\prime }\circ g_{\delta }-g_{\delta
}^{\prime }f_{0}^{\prime }\circ g_{0}||_{\infty }\rightarrow 0$ because of
the uniform continuity of \ $f_{0}^{\prime }$ and the uniform bound on $%
||g_{\delta }||_{C^{1}}$while $||g_{\delta }^{\prime }f_{0}^{\prime }\circ
g_{0}-g_{0}^{\prime }f_{0}^{\prime }\circ g_{0}||_{\infty }\rightarrow 0$
since $g_{\delta }\rightarrow g_{0}$ in $C^{1},$ and $(\ref{claim00})$ is
proved.

To prove $(\ref{claim0})$ we write
\begin{eqnarray*}
||f_{\delta }\circ g_{\delta }-f_{0}\circ g_{0}||_{\infty } &=&||f_{\delta
}\circ g_{\delta }-f_{0}\circ g_{\delta }+f_{0}\circ g_{\delta }-f_{0}\circ
g_{0}||_{\infty } \\
&=&||f_{\delta }\circ g_{\delta }-f_{0}\circ g_{\delta }||_{\infty
}+||f_{0}\circ g_{\delta }-f_{0}\circ g_{0}||_{\infty }
\end{eqnarray*}%
and note that%
\begin{equation*}
||f_{0}\circ g_{\delta }-f_{0}\circ g_{0}||_{\infty }\leq
Lip(f_{0})K_{1}\delta .
\end{equation*}%
while%
\begin{equation*}
||f_{\delta }\circ g_{\delta }-f_{0}\circ g_{\delta }||_{\infty }\leq
K_{1}\delta
\end{equation*}%
and $(\ref{claim0})$ is proved.
\end{proof}

We now prove Proposition \ref{mainprop}.
\begin{proof}[{Proof of Proposition \protect\ref{mainprop}} ]
We again denote by $\{y_{i}^{\delta }\}_{i=1}^{d}:=T_{\delta }^{-1}(x)$ and $%
\{y_{i}^{0}\}_{i=1}^{d}:=T_{0}^{-1}(x)$ the $d$ preimages under $T_{\delta }$
and $T_{0}$, respectively, of a point $x\in X$. We can write
\begin{eqnarray}
\frac{L_{\delta }w(x)-L_{0}w(x)}{\delta } &=&  \frac{1}{\delta }\left(
\sum_{i=1}^{d}\frac{w(y_{i}^{\delta })}{T_{\delta }^{\prime }(y_{i}^{\delta
})}-\sum_{i=1}^{d}\frac{w(y_{i}^{0})}{T_{0}^{\prime }(y_{i}^{0})}\right)
\label{eq1} \\
&=&\nonumber \underbrace{\frac{1}{\delta }\left( \sum_{i=1}^{d}w(y_{i}^{\delta
})\left( \frac{1}{T_{\delta }^{\prime }(y_{i}^{\delta })}-\frac{1}{%
T_{0}^{\prime }(y_{i}^{\delta })}\right) \right) }_{=:(I)}+\underbrace{\frac{%
1}{\delta }\left( \sum_{i=1}^{d}\frac{w(y_{i}^{\delta })-w(y_{i}^{0})}{%
T_{0}^{\prime }(y_{i}^{\delta })}\right) }_{=:(II)} \\
&&\nonumber +\underbrace{\frac{1}{\delta }\left( \sum_{i=1}^{d}w(y_{i}^{0})\left(
\frac{1}{T_{0}^{\prime }(y_{i}^{\delta })}-\frac{1}{T_{0}^{\prime
}(y_{i}^{0})}\right) \right) }_{=:(III)}.
\end{eqnarray}%
In the following we will analyze the terms $(I),(II),(III)$ showing the convergence of these terms to the three summands in the right hand side of \eqref{formulaz} both in the $C^1$ topology (see \eqref{58}) and in the uniform sense required by \eqref{GL6}. 


\emph{Term (I):} For the first term we first differentiate the expansion\footnote{Here and below we say $F_\delta$ is $O_{C^k}(\delta^2)$ if$
\underset{\delta \in (0,\overline{\delta })}{\sup }\frac{\|F_\delta\|_{C^k}}{\delta ^{2}}<+\infty$ for any $k\geq 0$.}  $T_{\delta
}(x)=T_{0}(x)+\delta \dot{T} (x)+O_{C^{2}}(\delta^2)$ in $x$ to get:
\begin{equation}\label{s86}
T_{\delta }^{\prime }(x)=T_{0}^{\prime }(x)+\delta \dot{T} ^{\prime
}(x)+O_{C^{1}}(\delta^2).
\end{equation}%


We can then write
\begin{eqnarray}
\nonumber (I) &=&\frac{1}{\delta }\left( \sum_{i=1}^{d}w(y_{i}^{\delta })\left( \frac{1%
}{T_{\delta }^{\prime }(y_{i}^{\delta })}-\frac{1}{T_{0}^{\prime
}(y_{i}^{\delta })}\right) \right) \\ \nonumber
&=&\frac{1}{\delta }\left( \sum_{i=1}^{d}\frac{w(y_{i}^{\delta })}{T_{\delta
}^{\prime }(y_{i}^{\delta })}\left( 1-\frac{T_{\delta }^{\prime
}(y_{i}^{\delta })}{T_{0}^{\prime }(y_{i}^{\delta })}\right) \right) \\ \nonumber
&=&\frac{1}{\delta }\left( \sum_{i=1}^{d}\frac{w(y_{i}^{\delta })}{T_{\delta
}^{\prime }(y_{i}^{\delta })}\left( 1-\left( \frac{T_{0}^{\prime
}(y_{i}^{\delta })+\delta \dot{T} ^{\prime }(y_{i}^{\delta
})+O_{C^{1}}(\delta^2)}{T_{0}^{\prime }(y_{i}^{\delta })}\right) \right)
\right) \\
&=&\left( -\sum_{i=1}^{d}\frac{w(y_{i}^{\delta })\dot{T} ^{\prime
}(y_{i}^{\delta })}{T_{\delta }^{\prime }(y_{i}^{\delta })T_{0}^{\prime
}(y_{i}^{\delta })}\right) +\frac{O_{C^{1}}(\delta^2)}{\delta}. \label{71}
\end{eqnarray}

Since $w(\cdot )\dot{T} ^{\prime }(\cdot )/(T_{\delta }^{\prime
}(\cdot )T_{0}^{\prime }(\cdot ))$ is $C^{1}$ on the circle  with uniformly
bounded norm when $\delta $ is small enough,  we have that
\begin{eqnarray}
\nonumber\lim_{\delta \rightarrow 0}\frac{1}{\delta }\left(
\sum_{i=1}^{d}w(y_{i}^{\delta }(\cdot ))\left( \frac{1}{T_{\delta }^{\prime
}(y_{i}^{\delta }(\cdot ))}-\frac{1}{T_{0}^{\prime }(y_{i}^{\delta }(\cdot ))}\right)
\right) &=&\lim_{\delta \rightarrow 0}\left( -\sum_{i=1}^{d}\frac{%
w(y_{i}^{\delta }(\cdot ))\dot{T} ^{\prime }(y_{i}^{\delta }(\cdot ))}{T_{\delta }^{\prime
}(y_{i}^{\delta }(\cdot ))T_{0}^{\prime }(y_{i}^{\delta }(\cdot ))}\right) \\
\label{termIlimit}&=&-L_{0}\left( \frac{w\dot{T} ^{\prime }}{T_{0}^{\prime }}\right),
\end{eqnarray}%
 with convergence in the $C^{1}$ topology, using Lemma \ref{tec},  $(\ref%
{claim00})$, and \eqref{s86} in the final equality.
This proves that $(I)$ converges in the $C^{1}$ topology to the first summand of \eqref{formulaz}.

We now turn to the uniform convergence in \eqref{GL6} for term (I). 
By uniform expansivity of $T_0$ and \eqref{s86} there is $K_{1}$ such that  
\begin{equation}\label{75}
\|\frac{w(\cdot
)\dot{T} ^{\prime }(\cdot )}{T_{\delta }^{\prime }(\cdot )T_{0}^{\prime
}(\cdot )}-\frac{w(\cdot )\dot{T} ^{\prime }(\cdot )}{(T_{0}^{\prime
}(\cdot ))^{2}}\|_{C^{0}}  
 \leq \|\frac{1}{T_{\delta }^{\prime }(\cdot)}-\frac{1}{T_{0}^{\prime}(\cdot )} \|_{C^0} \|\frac{w(\cdot )\dot{T} ^{\prime }(\cdot )}{T_{0}^{\prime}(\cdot )}\|_{C^{0}}
\leq K_{1}\delta 
\end{equation}
when $\delta $ is sufficiently small.
By the Sobolev Embedding Theorem,
$K_1$ is uniform for $w\in B_{W^{2,1}(0,1)}$.
We apply \eqref{75} and \eqref{claim0} to $L_\delta(w\dot{T}'/T_0')=\left( -\sum_{i=1}^{d}\frac{w(y_{i}^{\delta })\dot{T} ^{\prime
}(y_{i}^{\delta })}{T_{\delta }^{\prime }(y_{i}^{\delta })T_{0}^{\prime
}(y_{i}^{\delta })}\right)$ and $L_0(w\dot{T}'/T_0')=\left( -\sum_{i=1}^{d}\frac{w(y_{i}^{0})\dot{T} ^{\prime
}(y_{i}^{0})}{T_{0 }^{\prime }(y_{i}^{0})T_{0}^{\prime
}(y_{i}^{0})}\right)$
to obtain
\begin{equation}
\label{tocompare}
\|L_\delta(w\dot{T}'/T_0')-L_0(w\dot{T}'/T_0')\|_{C^0}=O(\delta).
\end{equation}
Note that the constant $K_{2}$ in $(\ref{claim0})$ only depends on $\mathrm{Lip}(f_{0})$ and $K_{1}$.

Substituting $L_\delta(w\dot{T}'/T_0')=\left( -\sum_{i=1}^{d}\frac{w(y_{i}^{\delta })\dot{T} ^{\prime
}(y_{i}^{\delta })}{T_{\delta }^{\prime }(y_{i}^{\delta })T_{0}^{\prime
}(y_{i}^{\delta })}\right)$, into \eqref{71}, we see that as $\delta \to 0$ one has $$\frac{1}{\delta }\left( \sum_{i=1}^{d}w(y_{i}^{\delta })\left( \frac{1%
}{T_{\delta }^{\prime }(y_{i}^{\delta })}-\frac{1}{T_{0}^{\prime
}(y_{i}^{\delta })}\right) \right)=L_\delta(w\dot{T}'/T_0') +\frac{O_{C^{1}}(\delta^2)}{\delta}.$$
By \eqref{tocompare}, we obtain
\begin{equation*}
\left\|\frac{1}{\delta }\left( \sum_{i=1}^{d}w(y_{i}^{\delta })\left( \frac{1}{%
T_{\delta }^{\prime }(y_{i}^{\delta })}-\frac{1}{T_{0}^{\prime
}(y_{i}^{\delta })}\right) \right) +L_{0}\left( \frac{w\dot{T} ^{\prime }}{%
T_{0}^{\prime }}\right) \right\|_{C^{0}}=O(\delta),
\end{equation*}%
and again by the Sobolev Embedding Theorem, and the uniformity of constants noted immediately after \eqref{tocompare} this is uniform for $w\in B_{W^{2,1}}(0,1)$.

\emph{Term (II):} We prove the convergence of  the second term of $(\ref{eq1})$ to the second summand of \eqref{formulaz} both in the sense of \eqref{58} and \eqref{GL6}.  Suppose $w\in C^{2}$. Using the Taylor expansion with Lagrange remainder   we have that 
\begin{eqnarray*}
w(y_{i}^{\delta }) &=&w(y_{i}^{0}+y_{i}^{\delta }-y_{i}^{0}) \\
&=&w(y_{i}^{0})+w^{\prime }(y_{i}^{0}) \lbrack y_{i}^{\delta
}-y_{i}^{0}]+\frac 12 w^{\prime \prime }(\xi)  \left(\lbrack y_{i}^{\delta
}-y_{i}^{0}]\right)^{2},
\end{eqnarray*}
where $\xi$ lies between $y_{i}^{\delta }$ and $y_{i}^{0}$.

%

Using Lemma \ref{tec} we get
\begin{equation*}
w(y_{i}^{\delta })=w(y_{i}^{0})+w^{\prime }(y_{i}^{0}) \left[-\delta \frac{%
\dot{T}(y_{i}^{0})}{T_{0}^{\prime }(y_{i}^{0})}+O_{C^{0}}(\delta
^{2})\right]+\frac 12 w^{\prime \prime }(\xi )\left(\delta \left( -\frac{\dot{T}(y_{i}^{0})}{%
T_{0}^{\prime }(y_{i}^{0})}\right) +O_{C^{0}}(\delta ^{2}))\right)^{2}.
\end{equation*}
Since $w''$ is uniformly bounded, 
 we get
\begin{equation*}
w(y_{i}^{\delta })=w(y_{i}^{0})+w^{\prime }(y_{i}^{0})-\delta \lbrack \frac{%
\dot{T}(y_{i}^{0})}{T_{0}^{\prime }(y_{i}^{0})}]+O_{C^0}(\delta^2).
\end{equation*}

%

Thus,
\begin{eqnarray*}
(II)=\frac{1}{\delta }\sum_{i=1}^{d}\frac{w(y_{i}^{\delta })-w(y_{i}^{0})}{%
T_{0}^{\prime }(y_{i}^{\delta })} &=&-\frac{1}{\delta }\sum_{i=1}^{d}\frac{%
\delta w^{\prime }(y_{i}^{0})\frac{\dot{T}(y_{i}^{0})}{T_{0}^{\prime
}(y_{i}^{0})}+O_{C^{0}}(\delta ^{2})}{T_{0}^{\prime }(y_{i}^{\delta })} \\
&=&-\sum_{i=1}^{d}\frac{\dot{T}(y_{i}^{0})w^{\prime }(y_{i}^{0})}{%
T_{0}^{\prime }(y_{i}^{0})T_{0}^{\prime }(y_{i}^{\delta })}+\frac{%
O_{C^{0}}(\delta ^{2})}{\delta }.
\end{eqnarray*}%
As in (\ref{termIlimit}), by Lemma \ref{tec}, one has 
in the $C^{0}$
topology%
\begin{equation}\label{IIlim}
\lim_{\delta \rightarrow 0}\frac{1}{\delta }\sum_{i=1}^{d}\frac{%
w(y_{i}^{\delta })-w(y_{i}^{0})}{T_{0}^{\prime }(y_{i}^{\delta })}%
=-L_{0}\left( \frac{\dot{T}w^{\prime }}{T_{0}^{\prime }}\right)(x).
\end{equation}
%

Now we prove the convergence in $C^{1}.$  
To begin this we first prove that in
the $C^{1}$ topology%
\begin{equation}
\lim_{\delta \rightarrow 0}\frac{1}{\delta }\sum_{i=1}^{d}\frac{%
w(y_{i}^{\delta })-w(y_{i}^{0})}{T_{0}^{\prime }(y_{i}^{\delta })}%
=\lim_{\delta \rightarrow 0}\frac{1}{\delta }\sum_{i=1}^{d}\frac{%
w(y_{i}^{\delta })-w(y_{i}^{0})}{T_{0}^{\prime }(y_{i}^{0})}.  \label{877}
\end{equation}
We note that 
\begin{eqnarray*}
\frac{1}{\delta }\sum_{i=1}^{d}[\frac{w(y_{i}^{\delta })-w(y_{i}^{0})}{%
T_{0}^{\prime }(y_{i}^{\delta })}-\frac{w(y_{i}^{\delta })-w(y_{i}^{0})}{%
T_{0}^{\prime }(y_{i}^{0})}] &=&\frac{1}{\delta }\sum_{i=1}^{d}\frac{%
[-T_{0}^{\prime }(y_{i}^{\delta })+T_{0}^{\prime
}(y_{i}^{0})][w(y_{i}^{\delta })-w(y_{i}^{0})]}{T_{0}^{\prime
}(y_{i}^{\delta })T_{0}^{\prime }(y_{i}^{0})} \\
&=&\frac{1}{\delta }\sum_{i=1}^{d}\frac{[\delta T_{0}^{\prime \prime
}(y_{i}^{0})\frac{\dot{T}(y_{i}^{0})}{T_{0}^{\prime }(y_{i}^{0})}%
+O_{C^{1}}(\delta ^{2})][w(y_{i}^{\delta })-w(y_{i}^{0})]}{T_{0}^{\prime
}(y_{i}^{\delta })T_{0}^{\prime }(y_{i}^{0})},
\end{eqnarray*}
using Lemma \ref{tec} for the second equality.
Since $w(y_{i}^{\delta })\rightarrow w(y_{i}^{0})$ in $C^{1}$ then by the expression immediately above, $\frac{1}{\delta }\sum_{i=1}^{d}[\frac{w(y_{i}^{\delta })-w(y_{i}^{0})}{%
T_{0}^{\prime }(y_{i}^{\delta })}-\frac{w(y_{i}^{\delta })-w(y_{i}^{0})}{%
T_{0}^{\prime }(y_{i}^{0})}]$ tends to $0$ in $C^{1}$ 
and \eqref{877} is proved.
Using this we now upgrade the convergence of $(\ref{IIlim})$ to $C^{1}$.
By \eqref{877} it is sufficient to prove%
\begin{equation*}
\lim_{\delta \rightarrow 0}\frac{d}{dx}[\frac{w(y_{i}^{\delta })-w(y_{i}^{0})%
}{\delta T_{0}^{\prime }(y_{i}^{0})}]=-\frac{d}{dx}\frac{\dot{T}%
(y_{i}^{0})w^{\prime }(y_{i}^{0})}{T_{0}^{\prime }(y_{i}^{0})^{2}}
\end{equation*}

in the $C^{0}$ topology. By uniform expansivity of $T_0$ this is equivalent to proving%
\begin{eqnarray}\nonumber
\lim_{\delta \rightarrow 0}\frac{d}{dx}[\frac{w(y_{i}^{\delta })-w(y_{i}^{0})%
}{\delta }] &=&\label{xxx1} -\frac{d}{dx}\frac{\dot{T}(y_{i}^{0})w^{\prime }(y_{i}^{0})}{%
T_{0}^{\prime }(y_{i}^{0})} \\
&=&-\frac{1}{T_{0}^{\prime }(y_{i}^{0})}\frac{[\dot{T}^{\prime
}(y_{i}^{0})w^{\prime }(y_{i}^{0})+w^{\prime \prime }(y_{i}^{0})\dot{T}%
(y_{i}^{0})]T_{0}^{\prime }(y_{i}^{0})-\dot{T}(y_{i}^{0})w^{\prime
}(y_{i}^{0})T_{0}^{\prime \prime }(y_{i}^{0})}{T_{0}^{\prime }(y_{i}^{0})^{2}%
}
\end{eqnarray}
in the $C^{0}$ topology. Thus let us compute%
\begin{eqnarray}
\nonumber\lim_{\delta \rightarrow 0}\frac{d}{dx}[\frac{w(y_{i}^{\delta })-w(y_{i}^{0})%
}{\delta }] &=&\lim_{\delta \rightarrow 0}\frac{\frac{1}{T_{\delta }^{\prime
}(y_{i}^{\delta })}w^{\prime }(y_{i}^{\delta })-\frac{1}{T_{0}^{\prime
}(y_{i}^{0})}w^{\prime }(y_{i}^{0})}{\delta } \\
\nonumber&=&\lim_{\delta \rightarrow 0}\frac{\frac{1}{T_{\delta }^{\prime
}(y_{i}^{\delta })}w^{\prime }(y_{i}^{\delta })-\frac{1}{T_{\delta }^{\prime
}(y_{i}^{\delta })}w^{\prime }(y_{i}^{0})}{\delta } \\
\label{split}&&+\lim_{\delta \rightarrow 0}\frac{\frac{1}{T_{\delta }^{\prime
}(y_{i}^{\delta })}w^{\prime }(y_{i}^{0})-\frac{1}{T_{0}^{\prime }(y_{i}^{0})%
}w^{\prime }(y_{i}^{0})}{\delta }.
\end{eqnarray}

To handle the first summand in \eqref{split} we remark that  
 by Taylor expansion with first-order Lagrange remainder we can obtain 
\begin{equation*}
w^{\prime }(y_{i}^{\delta })=
    [w^{\prime }(y_{i}^{0})+w^{\prime \prime }(y_{i}^{0})(y_{i}^{\delta
}-y_{i}^{0})]+{(w^{\prime \prime }(\xi )-w^{\prime \prime
}(y_{i}^{0}))(y_{i}^{\delta }-y_{i}^{0})}.
\end{equation*}
 with $\xi$ between $y_{i}^{0}$ and $y_{i}^{\delta}$.
 By the uniform continuity of $w''$ we get 
 $${(w^{\prime \prime }(\xi )-w^{\prime \prime
}(y_{i}^{0}))}\to 0$$
and then 
$${(w^{\prime \prime }(\xi )-w^{\prime \prime
}(y_{i}^{0}))(y_{i}^{\delta }-y_{i}^{0})}=o(\delta).$$ uniformly on the circle, then by Lemma \ref{tec} again
\begin{equation*}
w^{\prime }(y_{i}^{\delta })=w^{\prime }(y_{i}^{0})+w^{\prime \prime
}(y_{i}^{0})\delta \lbrack -\frac{\dot{T}(y_{i}^{0})}{T_{0}^{\prime
}(y_{i}^{0})}]+h_{2}(x,\delta )
\end{equation*}

with%
\begin{equation*}
||\frac{h_{2}(x,\delta )}{\delta }||_{C^{0}}\rightarrow 0.
\end{equation*}
By this we also have that $w(y_{i}^{\delta
})\rightarrow w(y_{i}^{0})$ in $C^{1}.$

Thus we get%
\begin{eqnarray}
\lim_{\delta \rightarrow 0}\frac{\frac{1}{T_{\delta }^{\prime
}(y_{i}^{\delta })}w^{\prime }(y_{i}^{\delta })-\frac{1}{T_{\delta }^{\prime
}(y_{i}^{\delta })}w^{\prime }(y_{i}^{0})}{\delta } &=&  \nonumber
\lim_{\delta
\rightarrow 0}\frac{1}{T_{\delta }^{\prime }(y_{i}^{\delta })}\frac{%
w^{\prime }(y_{i}^{\delta })-w^{\prime }(y_{i}^{0})}{\delta } \\
&=&  \nonumber
\lim_{\delta \rightarrow 0}\frac{1}{T_{\delta }^{\prime }(y_{i}^{\delta })%
}\frac{w^{\prime \prime }(y_{i}^{0})\delta \lbrack -\frac{\dot{T}(y_{i}^{0})%
}{T_{0}^{\prime }(y_{i}^{0})}]+h_{2}(x,\delta )}{\delta } \\
&=&\frac{-w^{\prime \prime }(y_{i}^{0})\dot{T}(y_{i}^{0})T_{0}^{\prime
}(y_{i}^{0})}{T_{0}^{\prime }(y_{i}^{0})^{3}}\label{a1}
\end{eqnarray}

in the $C^{0}$ topology. The second summand in \eqref{split} is estimated as in the proof of Lemma \eqref{tec} (see the calculations from  \eqref{for_later} to \eqref{for_later2}) obtaining

\begin{eqnarray*}
\lbrack \frac{1}{T_{\delta }^{\prime }(y_{i}^{\delta}(x))}-\frac{1}{T_{0}^{\prime
}(y_{i}^{0}(x))}] =-\frac{
y_{i}^{0}{}^{\prime }\dot{T}^{\prime }(y_{i}^{0})T_{0}^{\prime
}(y_{i}^{0})-y_{i}^{0}{}^{\prime }T_{0}^{^{\prime \prime }}(y_{i}^{0})\dot{T}%
(y_{i}^{0})}{(T_{0}^{\prime }(y_{i}^{0}))^{2}}+O_{C^0}(\delta^2).
\end{eqnarray*}

Then 
\begin{equation}
\lim_{\delta \rightarrow 0}\frac{\frac{1}{T_{\delta }^{\prime
}(y_{i}^{\delta })}w^{\prime }(y_{i}^{0})-\frac{1}{T_{0}^{\prime }(y_{i}^{0})%
}w^{\prime }(y_{i}^{0})}{\delta }=-w^{\prime }(y_{i}^{0})\frac{%
y_{i}^{0}{}^{\prime }\dot{T}^{\prime }(y_{i}^{0})T_{0}^{\prime
}(y_{i}^{0})-y_{i}^{0}{}^{\prime }T_{0}^{^{\prime \prime }}(y_{i}^{0})\dot{T}%
(y_{i}^{0})}{(T_{0}^{\prime }(y_{i}^{0}))^{2}}\label{a2}
\end{equation}

Thus putting together $(\ref{a1})$ and $(\ref{a2}),$ recalling that $(y_{i}^{0})'=\frac1{T'_0(y^0_i)}$, $(\ref{xxx1})$   is
proved. Thus we have proved the $C^1 $ convergence of the limit \eqref{IIlim}.
This proves that $(II)$ converges in the $C^{1}$ topology to the second summand of \eqref{formulaz}

Now we estimate the $L^1$ uniform convergence of term (II)
to prove the corresponding part of \eqref{GL6}. We now suppose $w\in W^{2,1}$. \ Using the
Taylor formula with integral remainder we have that 
\begin{eqnarray}
\nonumber w(y_{i}^{\delta }(x)) &=&w(y_{i}^{0}(x)+y_{i}^{\delta }(x)-y_{i}^{0}(x)) \\
\label{OL1delta2}&=&w(y_{i}^{0}(x))+w^{\prime }(y_{i}^{0}(x))\lbrack y_{i}^{\delta
}(x)-y_{i}^{0}(x)]+\int_{y_{i}^{0}(x)}^{y_{i}^{\delta }(x)}w^{\prime \prime
}(t)(y_{i}^{\delta }(x)-t)dt.
\end{eqnarray}
Let us estimate the $L^{1}$ norm of the function $x\to \int_{y_{i}^{0}(x)}^{y_{i}^{\delta
}(x)}w^{\prime \prime }(t)(y_{i}^{\delta }(x)-t)dt$%
\begin{equation*}
\int_{S^{1}}\left|\int_{y_{i}^{0}(x)}^{y_{i}^{\delta }(x)}w^{\prime \prime
}(t)(y_{i}^{\delta }(x)-t)dt\right|dx\leq \int_{S^{1}}\int_{\min
(y_{i}^{0}(x),y_{i}^{\delta }(x))}^{\max (y_{i}^{0}(x),y_{i}^{\delta
}(x))}|w^{\prime \prime }(t)|\cdot|(y_{i}^{\delta }(x)-t)|dtdx
\end{equation*}

we change the order of integration and get (see figure \ref{fig:int})

\begin{figure}
\centering
\includegraphics[width=0.4\textwidth]{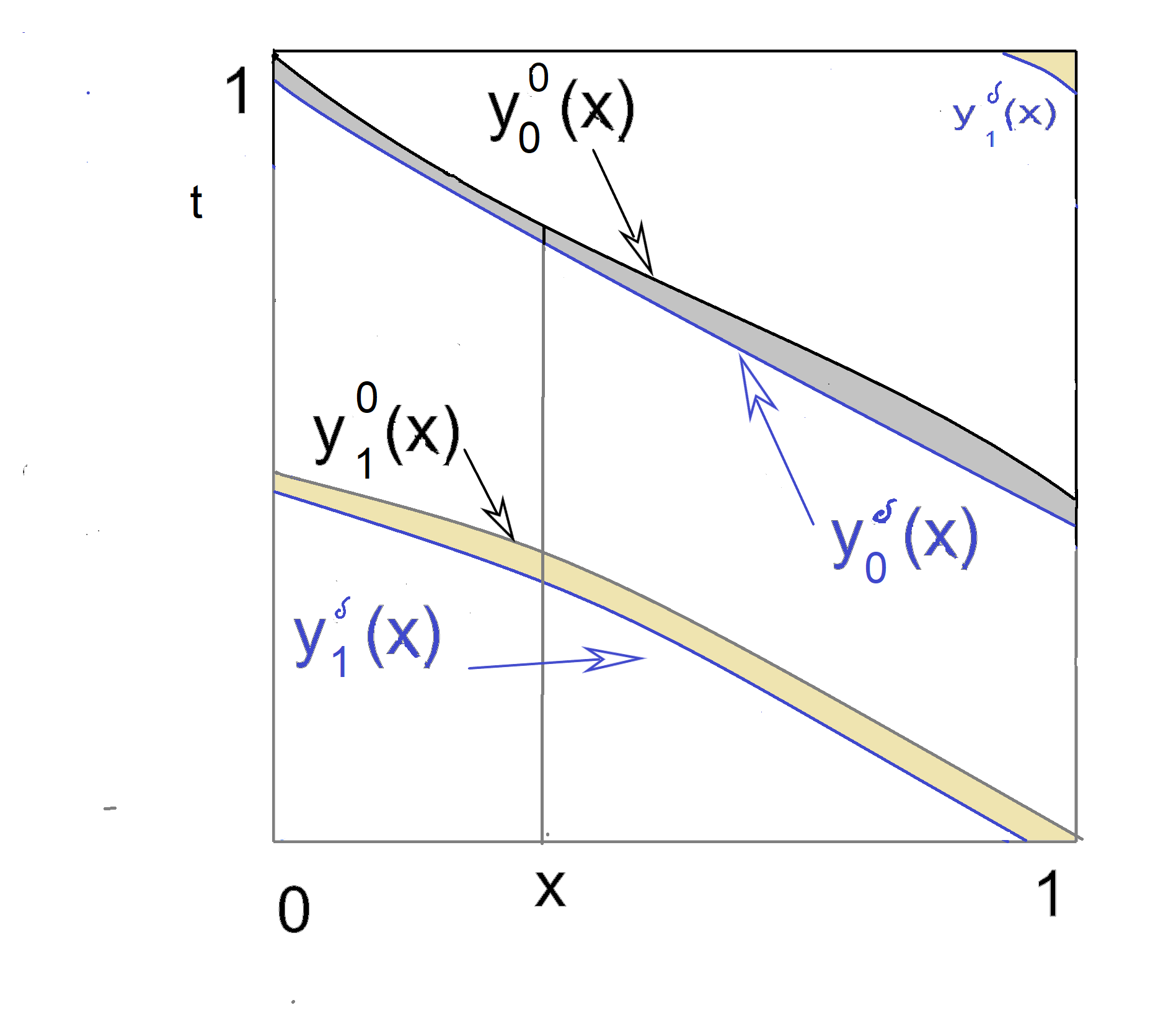}
\caption{
An example of the integration regions considered in \eqref{intfig}. The grey region 
 delimited by the inverse branches $y^0_0, y^\delta_0$ is the region on which the integral in the first term of \eqref{intfig} is defined, the union of the grey and the yellow region is the region (delimited by all inverse branches of the map before and after perturbation) on which which the integral the last term of \eqref{intfig} is defined.}
\label{fig:int}
\end{figure}

\begin{eqnarray}\label{intfig}
\int_{S^{1}}\int_{\min (y_{i}^{0}(x),y_{i}^{\delta }(x))}^{\max
(y_{i}^{0}(x),y_{i}^{\delta }(x))}|w^{\prime \prime }(t)||(y_{i}^{\delta
}(x)-t)| dt dx &\leq &\int_{S^{1}}\int_{\min (T_{0}(t),T_{\delta }(t))}^{\max
(T_{0}(t),T_{\delta }(t))}|w^{\prime \prime }(t)||(y_{i}^{\delta }(x)-t)| dx dt
\\ \nonumber
&=&\int_{S^{1}}|w^{\prime \prime }(t)|\int_{\min (T_{0}(t),T_{\delta
}(t))}^{\max (T_{0}(t),T_{\delta }(t))}|(y_{i}^{\delta }(x)-t)| dx dt.
\end{eqnarray}

Since $\delta\mapsto y_{i}^{\delta }(x)$ is uniformly Lipschitz and $%
|(y_{i}^{\delta }(x)-t)|=0$ \ at one of the two endpoints of the integration
interval $I_{\delta }(t):=[\min (T_{0}(t),T_{\delta }(t)),\max
(T_{0}(t),T_{\delta }(t))]$ we get that there is $K_{1}$ such that $%
\forall t$, $\max_{x\in I_{\delta }(t)}|(y_{i}^{\delta }(x)-t)|\leq K_{1}\delta $
and that there is $K_{2}$ such that $\forall t\ \int_{\min
(T_{0}(t),T_{\delta }(t))}^{\max (T_{0}(t),T_{\delta }(t))}|(y_{i}^{\delta
}(x)-t)|dx\leq K_{2}\delta ^{2}$. \ Thus 
\begin{equation}
\int_{S^{1}}\left|\int_{y_{i}^{0}(x)}^{y_{i}^{\delta }(x)}w^{\prime \prime
}(t)(y_{i}^{\delta }(x)-t)dt\right|dx \le K_2\delta^2\int_{S^1} |w''(t)|\ dt\leq K_{2}\delta ^{2}\|w\|_{W^{2,1}}.
\label{unifw31}
\end{equation}

Hence inserting the estimate from (\ref{unifw31}) into (\ref{OL1delta2}) we obtain%
\begin{equation*}
w(y_{i}^{\delta })=w(y_{i}^{0})-\delta w^{\prime }(y_{i}^{0})\frac{\dot{T}%
(y_{i}^{0})}{T_{0}^{\prime }(y_{i}^{0})}+O_{L^{1}}(\delta ^{2})
\end{equation*}%
where similarly as before  we write $g_{\delta }=O_{L^{1}}(\delta ^{2})$ 
when 
 $\sup_{\delta \in \lbrack 0,\overline{\delta })}\frac{||g_{\delta||_{L^1} }}{
\delta ^{2}}<+\infty $. \ We remark that the above
$O_{L^{1}}(\delta ^{2})$
term depends only on $w$ through 
 its $W^{2,1}$ norm, hence it is uniform as $w$ ranges in the unit
ball of such Sobolev space. 

Continuing, we write 
\begin{eqnarray}
(II)=\frac{1}{\delta }\sum_{i=1}^{d}\frac{w(y_{i}^{\delta })-w(y_{i}^{0})}{%
T_{0}^{\prime }(y_{i}^{\delta })} &=&-\frac{1}{\delta }\sum_{i=1}^{d}\frac{%
\delta w^{\prime }(y_{i}^{0})\frac{\dot{T}(y_{i}^{0})}{T_{0}^{\prime
}(y_{i}^{0})}+O_{L^{1}}(\delta ^{2})}{T_{0}^{\prime }(y_{i}^{\delta })} \\
&=&\label{p84}-\sum_{i=1}^{d}\frac{\dot{T}(y_{i}^{0})w^{\prime }(y_{i}^{0})}{%
T_{0}^{\prime }(y_{i}^{0})T_{0}^{\prime }(y_{i}^{\delta })}+\frac{%
O_{L^{1}}(\delta ^{2})}{\delta }.
\end{eqnarray}%
Now we recall that by Lemma \ref{tec},
$y_{i}^{\delta }=y_{i}^{0}+\delta \left( -\frac{\dot{T}(y_{i}^{0})}{%
T_{0}^{\prime }(y_{i}^{0})}\right) +O_{C^{1}}(\delta ^{2})
$,  furthermore $\frac{\dot{T}(\cdot)}{%
T_{0}^{\prime }(\cdot)T_{0}^{\prime }(\cdot)} $ is  Lipschitz, 
thus by 
  \eqref{claim0} we get that for all $1\leq i\leq d$ as $\delta \to 0$
\begin{equation}\label{e85-2}
||   \frac{\dot{T}(y_{i}^{0})}
{T_{0}^{\prime }(y_{i}^{0})T_{0}^{\prime }(y_{i}^{\delta })} -\frac{\dot{T}(y_{i}^{0})}{%
T_{0}^{\prime }(y_{i}^{0})T_{0}^{\prime }(y_{i}^{0})} ||_{C^0}=O(\delta).
\end{equation}
Since $w$ is in $W^{2,1}$, by the Sobolev Embedding Theorem $w'$ is continuous and there is $C\geq0$ such that  $||w'||_{C^0}\leq C ||w||_{W^{2,1}}.$
Thus there is a $C_2\geq0$ such that for sufficiently small $\delta$, for each $1\leq i\leq d$, one has
\begin{equation}\label{e85-1}
\sup_{w\in B_{W^{2,1}}(0,1)}||   \frac{\dot{T}(y_{i}^{0})w^{\prime }(y_{i}^{0})}
{T_{0}^{\prime }(y_{i}^{0})T_{0}^{\prime }(y_{i}^{\delta })} -\frac{\dot{T}(y_{i}^{0})w^{\prime }(y_{i}^{0})}{%
T_{0}^{\prime }(y_{i}^{0})T_{0}^{\prime }(y_{i}^{0})} ||_{C^0}=C_2\delta,
\end{equation}
and then
\begin{equation}\label{e85}
\sup_{w\in B_{W^{2,1}}(0,1)}||   \sum_{i=1}^{d}\frac{\dot{T}(y_{i}^{0})w^{\prime }(y_{i}^{0})}{%
T_{0}^{\prime }(y_{i}^{0})T_{0}^{\prime }(y_{i}^{\delta })} -\sum_{i=1}^{d}\frac{\dot{T}(y_{i}^{0})w^{\prime }(y_{i}^{0})}{%
T_{0}^{\prime }(y_{i}^{0})T_{0}^{\prime }(y_{i}^{0})} ||_{C^0}=dC_2\delta.
\end{equation}
By \eqref{p84} and \eqref{e85}  we get
\begin{equation*}
\sup_{w\in B_{W^{2,1}}(0,1)}||\frac{1}{\delta }\sum_{i=1}^{d}\frac{w(y_{i}^{\delta })-w(y_{i}^{0})}{%
T_{0}^{\prime }(y_{i}^{\delta })}+L_{0}\left( \frac{\dot{T}w^{\prime }}{%
T_{0}^{\prime }}\right) ||_{L^{1}}=O(\delta )
\end{equation*}
 uniformly as $w$ ranges in the unit ball of 
 ${W^{2,1}}$.

\emph{Term (III):} Finally, for the third term $(III)$ by Lemma \ref{tec} we can write
\begin{equation*}
T_{0}^{\prime }(y_{i}^{\delta })=T_{0}^{\prime }(y_{i}^{0})+T_{0}^{\prime
\prime }(y_{i}^{0})\left( -\frac{\dot{T} (y_{i}^{0})}{T_{0}^{\prime
}(y_{i}^{0})}\right) \delta +O_{C^{1}}(\delta^2).
\end{equation*}%

Again using Lemma \ref{tec} we get 
\begin{eqnarray*}
(III) &=&\frac{1}{\delta }\left( \sum_{i=1}^{d}w(y_{i}^{0})\left( \frac{1}{%
T_{0}^{\prime }(y_{i}^{\delta })}-\frac{1}{T_{0}^{\prime }(y_{i}^{0})}%
\right) \right) \\
&=&\frac{1}{\delta }\left( \sum_{i=1}^{d}w(y_{i}^{0})\left( \frac{%
T_{0}^{\prime }(y_{i}^{0})-T_{0}^{\prime }(y_{i}^{\delta })}{T_{0}^{\prime
}(y_{i}^{\delta })T_{0}^{\prime }(y_{i}^{0})}\right) \right) \\
&=&\frac{1}{\delta }\left( \sum_{i=1}^{d}w(y_{i}^{0})\left( \frac{%
T_{0}^{\prime }(y_{i}^{0})-\left( T_{0}^{\prime }(y_{i}^{0})+T_{0}^{\prime
\prime }(y_{i}^{0})\left( -\frac{\dot{T} (y_{i}^{0})}{T_{0}^{\prime
}(y_{i}^{0})}\right) \delta +O_{C^{1}}(\delta^2)\right) }{T_{0}^{\prime
}(y_{i}^{\delta })T_{0}^{\prime }(y_{i}^{0})}\right) \right) \\
&=&\left( \sum_{i=1}^{d}w(y_{i}^{0})\left( \frac{\dot{T}
(y_{i}^{0})T_{0}^{\prime \prime }(y_{i}^{0})}{T_{0}^{\prime
2}(y_{i}^{0})T_{0}^{\prime }(y_{i}^{\delta })}\right) \ \right)
+\frac{O_{C^{1}}(\delta^2)}{\delta}.
\end{eqnarray*}%
Finally, by Lemma \ref{tec} and $(\ref{claim00})$, similarly to what was done for \eqref{termIlimit},
 we get 
\begin{equation*}
\lim_{\delta \rightarrow 0}\frac{1}{\delta }\left(
\sum_{i=1}^{d}w(y_{i}^{0})\left( \frac{1}{T_{0}^{\prime }(y_{i}^{\delta })}-%
\frac{1}{T_{0}^{\prime }(y_{i}^{0})}\right) \right) =L_{0}\left( \frac{%
\dot{T} T_{0}^{\prime \prime }}{T_{0}^{\prime ^{2}}}w\right)
\end{equation*}%
in $C^{1}$. 
This proves that $(III)$ converges in the $C^{1}$ topology to the third summand of \eqref{formulaz}.

Furthermore, again with a reasoning similar to what done for \eqref{p84} (see the lines from \eqref{p84} 
 to \eqref{e85}) 
by Lemma \ref{tec} and \eqref{claim0}, and using arguments similar to those regarding uniformity of constants for Term (I) after equation \eqref{tocompare}, we obtain%
\begin{equation*}
\left\|\frac{1}{\delta }\left( \sum_{i=1}^{d}w(y_{i}^{0}(x))\left( \frac{1}{%
T_{0}^{\prime }(y_{i}^{\delta }(x))}-\frac{1}{T_{0}^{\prime }(y_{i}^{0}(x))}%
\right) \right) -L_{0}\left( \frac{\dot{T} T_{0}^{\prime \prime }}{%
T_{0}^{\prime 2}}w\right) (x)\right\|_{C^{0}}=O(\delta ),
\end{equation*}%
uniformly for $w\in B_{W^{2,1}}(0,1)$.
\end{proof}
\subsection{Proof of Proposition \ref{Prop:LRmain}}\label{endofapp}

In the following Lemma, we establish more precise regularity for the eigenvectors of the simple eigenvalues of the transfer operator.
\begin{lemma}\label{regv0}
Let~$L_0:W^{1,1}\rightarrow W^{1,1}$ be the transfer operator associated
to an expanding map $T_0\in C^{4}(S^{1})$. Let  $\lambda_{0}$ be a simple
eigenvalue of $L_0$ such that $|\lambda_0 |>\frac{1}{\inf (T_0^{\prime })}$,
let  $v_{0}\in W^{1,1}(S^{1})$ be the normalized eigenvector associated to $ \lambda_{0}$, as in (III) of Proposition \ref{Prop:LRmain}, then $v_{0}\in
W^{3,1}(S^{1})$.
\end{lemma}

\begin{proof}
Lemma A.3 of \cite{Ba3} 
(whose statement is recalled in the proof of Proposition \ref{eigdiff})
 can be applied to the transfer operator  $L_0:W^{1,1}\rightarrow
W^{1,1},$ obtaining that $v_{0}\in W^{2,1}$ as follows. 

By Lemma \ref{Lemsu} we get that $L_0$ satisfies Lasota--Yorke inequalities
both with $W^{1,1}$ and  $L^{1}$ and with $W^{2,1}$ and $W^{1,1}$
as a weak and strong spaces.  This implies that $L_0:W^{1,1}\rightarrow
W^{1,1}$ is continuous and preserves $W^{2,1}$. Furthermore the restriction of $L_0$
to $W^{2,1}$ is also continuous as an operator on $W^{2,1}.$  It is also well known that by   the Lasota--Yorke inequalities (Lemma \ref{Lemsu})
and the compact immersion between the strong and the weak spaces provided by the Rellich--Kondracov theorem, 
 the essential spectral radii of $L_0:W^{1,1}
\rightarrow W^{1,1}$ and $L_0:W^{2,1}\rightarrow W^{2,1}$ are smaller than $\frac{1}{\inf (T_0^{\prime})}$   (see e.g. Lemma 2.2 of [6]).
Thus Lemma A.3 of \cite{Ba3} can be applied, giving that the generator $v_0$ of
the one dimensional eigenspace associated to $\lambda _{0}$ is contained in $%
W^{2,1}$.
Now, since for a $C^4$ map, Lemma \ref{Lemsu} also gives a $W^{3,1}$, $W^{2,1}$ Lasota Yorke inequality, we can repeat the same reasoning using $W^{3,1}$, $W^{2,1}$ as a strong and weak space, obtaining that $v_0\in W^{3,1}$.

\end{proof}

\begin{proof}[Proof of Proposition \ref{Prop:LRmain}] 
\ 
{\color{red}}
Below we set $B_{ss}, B_s, B_w$ to be $W^{2,1}$, $W^{1,1}$, $L^{1}$,  respectively.

\textbf{Proof of Part I:}
The linear response formula for the invariant density ($\ref{linr}$) will follow as a direct
consequence of Theorem \ref{LR} and Proposition \ref{mainprop}.

To show that Theorem \ref{LR} can be applied we show that its assumptions (\ref{resex})--(\ref{reslim}) are verified for our maps and perturbations. 
Equation (\ref{resex}) will follow  from Proposition \ref{mainprop}.
Indeed, by (A0), $T_0\in C^4$, hence by Lemma \ref{regv0}  the normalized eigenvector $v_0$ of  $\lambda_0$ is contained in $W^{3,1}$ and hence is contained in $C^2.$

Furthermore, Propositions \ref{propora}  and \ref{resstab} provide  the existence and stability of the resolvent as required by the assumptions \eqref{resbound} and \eqref{reslim}, respectively, of Theorem \ref{LR}.

\textbf{Proof of Part II:}

The linear response formula $(\ref{lineig})$ for  isolated eigenvalues will be
established by applying Proposition \ref{lineig1}. Again, the necessary properties of the operator in this case are established in Proposition \ref{mainprop}.
To apply  Proposition \ref{lineig1} we  verify that the assumptions (\ref{resex}) and (GL1)--(GL7) hold for our maps and
perturbations. 
Assumption (\ref{resex}) has been established above in the proof of Part I.
The assumptions (GL1), (GL2)\ and (GL3) are verifed
by the Lasota--Yorke inequalities established in Lemma \ref{Lemsu}. 
The small perturbation assumptions needed at (GL4)  are established in Proposition\ \ref{prop14}.
Next we verify that
(GL5)--(GL7) are satisfied for perturbations of expanding maps satisfying (A0), (A1) and (A2). 

(GL5): We consider the derivative operator $\dot{L}$ as defined in \eqref{der}.  Since we consider the stronger and weaker spaces
 $B_{ss}, B_s, B_w$ to be {\color{red}}$W^{2,1}$, $W^{1,1}$, $L^{1}$, from \eqref{der} we directly verify that $\dot{L}:W^{1,1} \to L^1$ and  $\dot{L}:{W^{3,1}}\rightarrow W^{1,1}$ are bounded operators.
 
(GL6): We must show that
 $||L_{\delta }-L_{0}-\delta \dot{L}||_{W^{2,1}\rightarrow
L^1}\leq C\delta ^{2}$. This follows by multiplying \eqref{GL6} by $\delta$.

(GL7): By the Rellich-Kondracov theorems there are compact immersions between the spaces $W^{2,1}, W^{1,1}$, and $L^1$ and hence the assumption (GL7) is satisfied.

\end{proof}

\noindent \textbf{Acknowledgments.} GF thanks the University of Pisa's Department of Mathematics for its generous hospitality during a visit that initiated this research. 
GF is partially supported by an Australian Research Council Discovery Project. SG is partially supported by the
research project ``Stochastic properties of dynamical systems" (PRIN 2022NTKXCX) of the Italian Ministry of Education and Research.
\newline

\end{document}